\documentclass{article}

\usepackage{arxiv}

\usepackage[utf8]{inputenc} 
\usepackage[T1]{fontenc}    
\usepackage{hyperref}       
\usepackage{url}            
\usepackage{booktabs}       
\usepackage{amsfonts}       
\usepackage{nicefrac}       
\usepackage{microtype}      
\usepackage{lipsum}

\usepackage{url}
\usepackage{array,tabularx}
\usepackage{multicol, nonfloat}  
\usepackage{epsfig} 
\usepackage{amssymb}  
\usepackage{amsthm}
\newtheorem{theorem}{Theorem}[section]

\newtheorem{remark}{Remark}
\newtheorem{definition}{Definition}[section]
\newtheorem{proposition}[theorem]{Proposition}
\usepackage{blindtext}
\usepackage[utf8]{inputenc}	
\usepackage[english]{babel}
\usepackage{graphicx}
\usepackage{amsmath}
\usepackage{amsfonts}
\usepackage{epstopdf}
\usepackage{cite}
\usepackage{microtype}

\newtheorem{assumption}{Assumption}
\newtheorem{statement}{Statement}

\usepackage{color}
\newcommand{\norm}[1]{\| #1 \|}

\newcommand{\quotes}[1]{``#1''}

\title{Accelerated Distributed Primal-Dual Dynamics using Adaptive Synchronization}

\author{
  P. A. Bansode\thanks{P. A. Bansode is with department of Instrumentation Engineering, Ramrao Adik Institute of Technology, Mumbai, 400706 India.~{\tt\small prashant.bansode@rait.ac.in}} \\
   \And
 K. C. Kosaraju\thanks{K. C. Kosaraju is with faculty of Science and Engineering, University of Groningen, AG Groningen, 9747 The Netherlands.}  \\
 \And
 S. R. Wagh\thanks{S. R. Wagh is with department of Electrical Engineering, Veermata Jijabai Technological Institute, Mumbai, 400019 India.} \\
 \And
 R. Pasumarthy\thanks{R. Pasumarthy is with department of Electrical Engineering, Indian Institute of Technology Madras, Madras, 600036 India.} \\
 \And
 N. M. Singh\thanks{N. M. Singh is with department of Electrical Engineering, Veermata Jijabai Technological Institute, Mumbai, 400019 India.} 
}

\begin{document}
\maketitle

\begin{abstract}
This paper proposes an adaptive primal-dual dynamics for distributed optimization in multi-agent systems. The proposed dynamics incorporates an adaptive synchronization law that reinforces the interconnection strength between the coupled agents. By strengthening the synchronization between the primal variables of the coupled agents, the given law accelerates the convergence of the proposed dynamics to the saddle-point solution. The resulting dynamics is represented as a feedback-interconnected networked system that proves to be passive. The passivity properties of the proposed dynamics are exploited along with the LaSalle’s invariance principle for hybrid systems, to establish asymptotic convergence and stability of the saddle-point solution. Further, the primal dynamics is analyzed for the rate of convergence and stronger convergence bounds are established, it is proved that the primal dynamics achieve accelerated convergence under the adaptive synchronization. The robustness of the proposed dynamics is quantified using $L_2$-gain analysis and the correlation between the rate of convergence and robustness of the proposed dynamics is presented. The effectiveness of the proposed dynamics is demonstrated by applying it to solve distributed least squares and distributed support vector machines problems.
\end{abstract}


\section{Introduction}
Distributed optimization techniques have been the subject of substantial research for many years. Their applications include wireless sensor networks \cite{rabbat2004distributed,johansson2008distributed,bertrand2011consensus}, power networks \cite{yi2015distributed}, large scale support vector machines \cite{forero2010consensus,stolpe2016distributed} etc. An exhaustive survey of these techniques can be found in \cite{nedic2018distributed}. Mainly, distributed optimization techniques are categorized as either decomposition based distributed optimization (see, \cite{palomar2006tutorial} and references therein) or consensus-based distributed optimization. The consensus based distributed optimization techniques have been significantly explored lately \cite{nedic2010constrained,forero2010consensus,yi2015distributed,stolpe2016distributed,lin2017distributed,sakurama2017distributed,wang2019distributed}, which is the prime subject of this paper.

Many algorithms have been proposed to solve consensus-based distributed optimization problems arising in networked systems, such as the seminal work on distributed sub-gradient methods\cite{nedic2009distributed}, distributed primal-dual dynamical algorithms\cite{yi2015distributed}, distributed gradient descent algorithms \cite{qu2016accelerated,lin2017distributed} etc. Out of these, the distributed primal-dual dynamics based algorithms deserve special attention because of their rich systems and control theoretic properties \cite{feijer2010stability,wang2011control,cherukuri2016asymptotic,simpson2016input,kosaraju2018stability} and ability to obtain simultaneously both primal as well as dual optimal solutions. The seminal work on the primal-dual dynamics or the saddle point dynamics dates back to late $1950$s \cite{kose1956solutions,arrow1958studies}. While their application for solving optimization problems over a network first appeared in \cite{feijer2010stability} with the focus on asymptotic convergence and stability of these algorithms. This framework is later extended to distributed optimization over a network of communicating nodes in \cite{droge2014continuous,yi2015distributed}. The primal-dual dynamics in \cite{droge2014continuous} combine the decomposition and the consensus-based methods to propose proportional-integral distributed optimization for equality constrained optimization problems and achieves a globally asymptotically stable saddle-point solution. The primal-dual gradient-based algorithm proposed in \cite{yi2015distributed} achieves asymptotic convergence for a consensus-based distributed optimization problem with local inequality constraints and implements the algorithm for load-sharing control in power networks. The notion of asymptotic convergence and stability of the (distributed) primal-dual dynamics for distributed optimization has been well established.

From the perspectives of online optimization, it is necessary to certify the distributed optimization algorithms on the basis of their rate of convergence as well as stability. The asymptotic convergence of the primal-dual dynamics implies that the trajectories will converge to the saddle-point solution as $t \rightarrow \infty$ which is not sufficient as a notion when the algorithm solves the distributed optimization problem online. Lately, the algorithms such as distributed gradient (sub-gradient) methods have been widely re-studied with the objective of improvement in the rate of convergence, see \cite{qu2016accelerated,7912404,nedic2017achieving,lin2017distributed,nedic2018improved}. However, the distributed primal-dual dynamics are not yet explored with the same objective which could limit their application to large-scale distributed optimization problems. While the existing methods on improving the rate convergence of the primal-dual dynamics rely upon increasing the convexity of the objective function by using quadratic penalty terms (augmented Lagrangian techniques)\cite{simpson2016input}, their usage for solving distributed optimization problems will destroy the distributed structure of the objective function. Thus, increasing convexity by using quadratic penalties may not pose as a suitable way of improving the rate of convergence of the distributed primal-dual dynamics. As an alternative route to this could be to exploit the graph-Laplacian properties of the underlying network and use adaptive coupling gains between the nodes to improve the convergence results. Addressing this issue, the present work primarily contributes to the accelerated convergence of the distributed primal-dual dynamics.

\subsection{Relevant literature and contributions}
The work proposed in this paper is in the same spirit with the recent articles \cite{yi2015distributed,kosaraju2018stability}. In \cite{yi2015distributed}, the framework of primal-dual dynamics for network utility maximization \cite{feijer2009krasovskii} which uses Krasovskii type Lyapunov function to derive asymptotic convergence, is extended for distributed optimization with application to load sharing control in power systems. Our contribution significantly differs from \cite{yi2015distributed} in the sense that the proposed dynamics is first analyzed using passivity tools of dynamical systems which then lead to its asymptotic stability when combined with the LaSalle's invariance principle of hybrid systems \cite{lygeros2003dynamical}. The advantage of passivity-based stability analysis is that the proposed dynamics can be realized as a negative feedback interconnection of the primal and the dual subsystems. This also facilitates to understand the interaction between the primal and the dual dynamical subsystems through their inputs and outputs. Thus each subsystem also enjoys $L_2$ stability properties of feedback connected dynamical systems. This feature later comes to the aid of robustness analysis of the proposed dynamics using $L_2$-gains. The fundamental results on passivity-based stability analysis of the primal-dual dynamics are established in \cite{kosaraju2018stability}. Our work, in a way, extends these results for the consensus-based distributed optimization problems. However, the primal dynamical subsystem derived in this paper does not use Brayton-Moser framework \cite{brayton1964theory} to arrive at the optimal solution. 

The central theme of the paper, that is the adaptively coupled primal-dual dynamics is derived by integrating the consensus protocol in the distributed primal-dual dynamics with the adaptive coupling laws motivated from the results in \cite{shafi2015adaptive}. In \cite{shafi2015adaptive}, the adaptive synchronization technique has been proved to guarantee the synchronization between the trajectories of diffusively coupled agents of a multiagent system. This technique is essentially based upon modifying the coupling weights of the diffusively coupled agents in accordance with the synchronization error between them. Larger values of synchronization errors result in increasing the coupling weights and vice-a-versa. In this paper, it is shown that the adaptation in the coupling weights strengthens the synchronization of the primal variables of the coupled agents. With this, the proposed work establishes results on an accelerated convergence of the proposed dynamics to the saddle point solution. While the adaptive synchronization has proved to accelerate the convergence, it is shown that it affects the robustness of the proposed dynamics. By introducing exogenous inputs in the interconnected network dynamics of the primal-dual subsystems, the $L_2$-gain of the proposed dynamics is analyzed and worst case $L_2$-gain is quantified in correlation with the rate of convergence. Although it is well known that the interconnected network of passive dynamical systems is inherently robust to exogenous inputs \cite{van2000l2}, our results have quantified the $L_2$-gain margins and established a relation between these margins and the rate of convergence. 

To summarize, the proposed work envelopes the following key points:
\begin{enumerate}
	\item The proposed algorithm, designated hereafter as the adaptively synchronized distributed primal-dual dynamics (ADPDD), ensures synchronization of the network-wide primal variables to a common trajectory which is then driven to the optimal solution.
	\item The ADPDD is posed as a negative feedback interconnection of the primal dynamical subsystem and the dual dynamical subsystems. It is proved that these subsystems remain individually passive, which subsequently, ensures the passivity and the asymptotic stability of the proposed dynamics.
	\item The convergence rate of the ADPDD is established and it is proved that the ADPDD has an accelerated convergence than the distributed primal-dual dynamics (DPDD).
	\item By allowing time-scale separation between the adaptive coupling laws and the primal-dual dynamics, the adaptively coupled distributed primal dynamics is proved to have an accelerated convergence to the optimal solution than the conventional distributed primal dynamics.
	\item The $L_2$-gain analysis of the proposed dynamics against the exogenous disturbances is presented to show the correlation between the rate of convergence and the robustness of the proposed algorithm.
	\item Applicability of the proposed algorithm to solve distributed least squares, distributed support vector machines problems is discussed.
\end{enumerate}    

\subsection{Notations and Preliminaries}\label{prel}
The set $\mathbb{R}$ (respectively $\mathbb{R}_{\geq 0}$ or $\mathbb{R}_{> 0}$) is the set of real (respectively non-negative or positive) numbers. $I_n$ is the $n \times n$ identity matrix. $\mathbf{0}$ is a zero vector of appropriate dimensions.
For a square matrix $A \in \mathbb{R}^{n \times n}$, $\mathrm{eig}(A)=\{\lambda_1(A),\lambda_2(A),\ldots,\lambda_n(A)\} \in \mathbb{R}$ represents eigenvalues of $A$ in an ascending order. The smallest eigenvalue of $A$ is given by $\lambda_1(A)$ and the second smallest eigenvalue is given by $\lambda_2(A)$. If  $B \in \mathbb{R}^{m\times n}$ and $C \in \mathbb{R}^{p \times q}$ are real matrices, then $B \otimes C \in \mathbb{R}^{mp \times nq}$ is a block matrix that defines the Kronecker product of $B$ and $C$. 

The interaction topology in a multi-agent system is represented using an undirected graph $\mathcal{G}=(\mathcal{N}, \mathcal{E})$ with $\mathcal{N}= \{1, 2, \ldots, n\}$ as the set of agents and $\mathcal{E} \subseteq \mathcal{N} \times \mathcal{N}$ as the set of edges. The neighbor set of the $i^{th}$ agent is $\mathcal{N}_i=\{q\in \mathcal{N}| (q,i) \in \mathcal{E}\}$, where $i \in \mathcal{N}$. The number of agents $n$ is the cardinality of $\mathcal{G}$. Let $D \in \mathbb{R}^{n \times n}$ be the degree matrix of $\mathcal{G}$ and $A \in \mathbb{R}^{n \times n}$ be the adjacency matrix of $\mathcal{G}$, with elements $a_{iq}=a_{qi}>0,\forall (i,q) \in \mathcal{E}$, then $L=D-A$ is the Laplacian matrix of $\mathcal{G}$. By definition, $L\in \mathbb{R}^{n \times n}$ is a symmetric positive semidefinite matrix that encodes the connectivity of the agents and their interaction topology in $\mathcal{G}$.

If $f: \mathbb{R}^n \rightarrow \mathbb{R}$ is continuously differentiable in $x \in \mathbb{R}^n$, then $\nabla_x f:\mathbb{R}^n \rightarrow \mathbb{R}^n$ is the gradient of $f$ with respect to $x$. If $f$ is twice continuously differentiable and strictly convex in $x$ then $\mathbb{H}=\nabla^2_xf \in \mathbb{R}^{n \times n}_{>0}$ is a symmetric positive definite matrix of second-order partial derivatives of $f$ with respect to $x$.

Consider the following dynamical system 
\begin{align}
\dot{x}=F(x,u),y=G(x,u),\label{dynamic}
\end{align} where state $x \in \mathbb{R}^n$, input $u \in \mathbb{R}^m$, and output $y \in \mathbb{R}^m$, with $F,G$ (of appropriate dimensions) sufficiently smooth and satisfying $F(0)=G(0)=0$. 
\begin{definition}[\cite{khalil1996noninear}]
	The system \eqref{dynamic} is said to be passive if there exists a positive semidefinite storage function (Lyapunov function) $V: \mathbb{R}^n \rightarrow \mathbb{R}$, continuously differentiable in $x$ such that $\dot{V}\leq u^Ty$.\end{definition}
For scalars $x,y$, $[x]^+_y:=x$ if $y>0$ or $x>0$, and $[x]^+_y:=0$ otherwise. 

The remainder of the paper is mainly divided into two sections. Section \ref{pf} discusses the main results of the paper and Section \ref{sims} presents examples to validate the proposed work. Subsection \ref{pf} is divided as follows: Section \ref{sec2a} describes the consensus-based distributed optimization problem. In Subsection \ref{adsy} the adaptive synchronization technique is elaborated. Subsection \ref{adpdd} formulates the adaptive distributed primal-dual dynamical algorithm to solve distributed optimization problem proposed in Subsection \ref{sec2a}. Subsections \ref{passive_pd} and \ref{gb} present passivity and stability analysis of the proposed dynamics. In Subsection \ref{fasterc} the convergence bounds of the proposed algorithm are obtained and the proof for an accelerated convergence of the same is provided. Subsection \ref{l2stab} provides $L_2$-gain analysis of the proposed dynamics and establishes a correlation between both robustness and rate of convergence of the same. Section \ref{sims} presents the application of the proposed dynamics to the distributed least squares and the distributed support vector machines problems. Some numerical examples of academic interests are also discussed. Section \ref{concl} concludes the paper.

\section{Problem Formulation and main results}\label{pf}
\subsection{Distributed Optimization}\label{sec2a}
Consider the following distributed optimization problem
\begin{align} 
\begin{aligned}
\min_{x \in \mathbb{R}^{ln}}&~f(x) = \sum_{i=1}^{n}f_i(x_i)\\
\mathrm{subject~to}&~x_{ik}=x_{qk},~\forall^l_{k=1},\forall i,q \in \mathcal{N},\\
&~g_j(x_{ik})\leq 0,~\forall^{m^{ik}_g}_{j = 1},\forall^l_{k=1},\forall i \in \mathcal{N},
\end{aligned}\label{diopt}
\end{align}
where $x_i = [x_{i1},\ldots,x_{il}]^T\in \mathbb{R}^l$ and $x = [x^T_1,\ldots,x^T_n]^T \in \mathbb{R}^{ln}$. It is assumed that the functions $f_i:\mathbb{R}^l \rightarrow \mathbb{R}$ is twice differentiable and strongly convex, and $g_j:\mathbb{R} \rightarrow \mathbb{R}$ is convex. The optimization problem \eqref{diopt} can be decomposed into $n$ subproblems wherein each subproblem minimizes the cost $f_i(x_i)$ subject to the consensus constraint $x_{ik}=x_{qk}$ and inequality constraints $g_j(x_{ik})\leq 0$.
The problem \eqref{diopt} can not be fully decoupled into a set of $n$ subproblems because of the consensus constraints, but it can be addressed as a network-based multiagent optimization problem using graph theory as a tool. Let an undirected and connected graph $\mathcal{G}(\mathcal{N},\mathcal{E})$ describe the communication topology of the underlying network, where $\mathcal{N}$ denotes the set of agents or subproblems, and $\mathcal{E}$ denotes the set of communication links. Each agent minimizes a local cost function $f_i(x_i)$ subject to the consensus constraints $x_{ik}=x_{qk},\forall^l_{k=1},\forall q \in \mathcal{N}_i$ and the local inequality constraints $g_j(x_{ik})\leq 0,\forall^l_{k=1}$.  The global consensus corresponds to the optimal solution of \eqref{diopt}, when $x^{*}_1=x^{*}_2=\ldots=x^*_n=x^*$. The index $m^{ik}_g$ is the number of inequality constraints associated with the scalar $x_{ik}$. 

The strong duality of \eqref{diopt} is subject to the convexity of $f$ and the constraint satisfaction given by the Slater's condition (see, \cite{boyd2004convex}), which is as follows:
Assuming that there exists an $x \in \mathrm{relint}\mathcal{D}$ such that $g_j(x_{ik})< 0,x_{ik}=x_{qk},\forall^l_{k=1},\forall q \in \mathcal{N}_i,\forall^n_{i=1}$, then $x$ is strictly feasible, where $\mathcal{D}$ is the domain of \eqref{diopt} defined as $\mathcal{D}=\mathrm{dom}f$.
The convexity of $f$ strongly imply the uniqueness of its optimal solution $x^*$.
\begin{assumption}\label{asss} $f$ is $\mu-$strongly convex  and $\kappa$-smooth, i.e., for all $x_1,x_2\in \mathbb{R}^{ln}$,
	\begin{align}
	\mu \norm{x_1-x_2}^2\leq \langle \nabla f(x_1)-\nabla f(x_2),x_1-x_2\rangle \leq \kappa\norm{x_1-x_2}^2.\label{lipsch}
	\end{align}
\end{assumption}
The Lagrangian function ${\mathcal{L}}$ of the problem \eqref{diopt} is given by:
\begin{align}
\begin{aligned} 
{\mathcal{L}}(x,\alpha,\theta)&=f(x)+\alpha^T(L\otimes I_{l})x+\sum_{i=1}^{n}\sum_{k=1}^{l}\sum_{j=1}^{m^{ik}_g}\theta^{ik}_j g_j(x_{ik}), 
\end{aligned}\label{dlag1}
\end{align}
where $\alpha_{ik}\in \mathbb{R}$ is a Lagrange multiplier associated with the consensus constraint $x_{ik}=x_{qk}$ and $\theta^{ik}_j\in \mathbb{R}_+=\{\theta^{ik}_j \in \mathbb{R}| \theta^{ik}_j\geq 0,~\forall^l_{k=1},\forall^{m^{ik}_g}_{j=1}, \forall i \in \mathcal{N}\}$ is a Lagrange multiplier associated with the inequality constraint $g_j(x_{ik})\leq 0$. The vector notations of the respective Lagrange multipliers are $\theta\in \mathbb{R}^{m^{ik}_gln}_+$ and $\alpha \in \mathbb{R}^{ln}$.



\begin{remark}
	Assuming that the Slater's condition is satisfied and a strong duality holds, the saddle-point $(x^*,\alpha^*,\theta^*)$ satisfies the Karush-Kuhn-Tucker (KKT) conditions derived the Lagrangian \eqref{dlag1}, as follows:
	\begin{align}
	&\nabla_{x^*_{ik}}f_i(x^*_{ik})+\sum_{q\in \mathcal{N}_i} a_{iq}(\alpha^*_{ik}-\alpha^*_{qk})\nonumber\\
	&+\sum_{j=1}^{m^{ik}_g}(\theta^{ik}_j)^*\nabla_{x^*_{ik}}g_j(x^*_{ik})=0,\forall^{l}_{k=1},\forall i \in \mathcal{N},\nonumber\\
	&g_j(x^*_{ik})\leq 0,(\theta^{ik}_j)^*\geq 0,\forall^{l}_{k=1},\forall^{m^{ik}_g}_{j=1},\forall i \in \mathcal{N},\nonumber\\
	&(\theta^{ik}_j)^\ast g_j(x^*_{ik})=0,\forall^{l}_{k=1},\forall^{m^{ik}_g}_{j=1},\forall i \in \mathcal{N},\nonumber\\ 
	&x^*_{ik}=x^*_{qk},~\forall^l_{k=1},\forall (i,q) \in \mathcal{N}.
	\label{dkkt}
	\end{align}
\end{remark}

In order to ensure the global consensus of the states $x_i,\forall i \in \mathcal{N}$, the Lagrangian function defined in \eqref{dlag1} is augmented with the term $x^T(L\otimes I_{l})x$. The augmented Lagrangian function is defined below:
\begin{align}
\begin{aligned} 
\bar{\mathcal{L}}(x,\alpha,\theta)={\mathcal{L}}(x,\alpha,\theta)+x^T(L\otimes I_{l})x.
\end{aligned}\label{dlag}
\end{align}
\begin{remark}
	Note that augmenting the Lagrangian \eqref{dlag1} with $x^T(L\otimes I_{l})x$ does not affect its convexity-concavity properties. This owes to the fact that $x^T(L\otimes I_{l})x$ is a positive semidefinite function of the primal variable $x$. Thus the saddle-point satisfying \eqref{dkkt} also satisfies the following KKT conditions for the Lagrangian \eqref{dlag}:
	\begin{align}
	&\nabla_{x^*_{ik}}f_i(x^*_{ik})+\sum_{q\in \mathcal{N}_i} a_{iq}(x^*_{ik}-x^*_{qk})+\sum_{q\in \mathcal{N}_i} a_{iq}(\alpha^*_{ik}-\alpha^*_{qk})\nonumber\\
	&+\sum_{j=1}^{m^{ik}_g}(\theta^{ik}_j)^*\nabla_{x^*_{ik}}g_j(x^*_{ik})=0,\forall^{l}_{k=1},\forall i \in \mathcal{N},\nonumber\\
	&g_j(x^*_{ik})\leq 0,(\theta^{ik}_j)^*\geq 0,\forall^{l}_{k=1},\forall^{m^{ik}_g}_{j=1},\forall i \in \mathcal{N},\nonumber\\
	&(\theta^{ik}_j)^\ast g_j(x^*_{ik})=0,\forall^{l}_{k=1},\forall^{m^{ik}_g}_{j=1},\forall i \in \mathcal{N},\nonumber\\ 
	&x^*_{ik}=x^*_{qk},~\forall^l_{k=1},\forall (i,q) \in \mathcal{N}.
	\label{dkkt1}
	\end{align}\label{kkta}
\end{remark}

Using the augmented Lagrangian \eqref{dlag}, the primal-dual dynamics is derived as follows:
\begin{align} 
\begin{aligned}
\dot{x}_{ik}&=-\nabla_{x_{ik}}\bar{\mathcal{L}}(x,\alpha,\theta), \dot{\alpha}_{ik}=\nabla_{\alpha_{ik}} \bar{\mathcal{L}}(x,\alpha,\theta),\\
\dot{\theta}^{ik}_j&=[\nabla_{\theta^{ik}_j}\bar{\mathcal{L}}(x,\alpha,\theta)]^+_{\theta^{ik}_j},~\forall^l_{k=1};\forall^{m^{ik}_g}_{j=1}; \forall i \in \mathcal{N}.
\end{aligned}\label{pddy}
\end{align}
With the primal-dual dynamics derived as given in \eqref{pddy}, the following subsection develops the ADPDD.
\subsection{Adaptively Synchronized Distributed Primal-dual dynamics}\label{adpddwow}
The following subsection presents the adaptive synchronization mechanism which is later integrated with the dynamics defined in \eqref{pddy} to arrive at ADPDD.
\subsubsection{Adaptive synchronization}\label{adsy}
The adaptive synchronization mechanism has been widely used in multi-agent systems to guarantee synchronization between the agents with respect to their state variables \cite{shafi2015adaptive,delellis2009novel}, which is explained subsequently.

The primal variables associated with each agent evolve according to
\begin{align}
\dot{x}_{ik}=-\nabla_{x_{ik}}\bar{\mathcal{L}}(x,\alpha,\theta) \label{pd1}
\end{align} 
as described in \eqref{pddy}.
By performing gradient descent on \eqref{dlag}, the primal dynamics \eqref{pd1} can be further derived as:
\begin{align}
\dot{x}_{ik}&=-\nabla_{x_{ik}}f(x)-\sum_{q\in \mathcal{N}_i} a_{iq}(x_{ik}-x_{qk})\nonumber\\
&~~~-\sum_{q\in \mathcal{N}_i} a_{iq}(\alpha_{ik}-\alpha_{qk})-\sum_{j=1}^{m^{ik}_g}\theta^{ik}_j\nabla_{x_{ik}}g_j(x_{ik}).\label{pd2}
\end{align}
Let $u_{x_{ik}} \in \mathbb{R}$ corresponds to the following term in \eqref{pd2}:
\begin{align} 
\begin{aligned}
u_{x_{ik}}=-\sum_{q\in \mathcal{N}_i}a_{iq}(x_{ik}-x_{qk}),\forall q \in \mathcal{N}_i,
\end{aligned}\label{conse}
\end{align} where the interconnection strength or the coupling weight $a_{iq}$ belongs to the adjacency matrix $A$ such that
\begin{align} 
a_{iq}=a_{qi}=\begin{cases}
\mathrm{a~positive~scalar,}~~\mathrm{for}~(q,i) \in \mathcal{E},\\
0,~~\mathrm{for}~(q,i) \notin \mathcal{E}.
\end{cases}\nonumber
\end{align}
The equation \eqref{conse} is regarded widely as the consensus protocol or the consensus law\cite{shafi2015adaptive,mesbahi2010graph}.
Define further $u_{x_{i}} \in \mathbb{R}^l$, the consensus protocol \eqref{conse} can be modified to accommodate $x_i\in \mathbb{R}^l$ as given below:
\begin{align} 
\begin{aligned}
u_{x_{i}}=-\sum_{q\in \mathcal{N}_i}a_{iq}(x_{i}-x_{q}),\forall q \in \mathcal{N}_i,
\end{aligned}\label{conse1}
\end{align} Similarly, 
\begin{align}
u_x = -(L \otimes I_l)x \label{ux}
\end{align} is a compact form representation of \eqref{conse1}.

If $i$ and $q$ are neighbors in $\mathcal{G}$ with $e_{iq} = x_i-x_q$ defined as the local synchronization error, then the coupling weight can be represented as a function of $e_{iq}$, i.e. $\dot{a}_{iq}=h_i(e_{iq})$, where $h_i:\mathbb{R}^l \rightarrow \mathbb{R}$  monotonically increases in $e_{iq}$. It yields a stronger synchronization between the primal variables of the coupling agents which motivates to incorporate adaptive synchronization to address the convergence rate of the distributed primal-dual dynamics. In line with this, the following coupling weight update rule is proposed:
\begin{align} 
\dot{a}_{iq}=d_{iq}(e^T_{iq}e_{iq}+\dot{e}^T_{iq}\dot{e}_{iq}),\label{ajl}
\end{align} where $d_{iq}=d_{qi}>0$ is the adaptive gain constant. 
\begin{remark}
	Represent \eqref{ajl} in the form $\dot{a}_{iq}=h_i(e_{iq},\dot{e}_{iq})$, throughout the rest of the paper it is assumed that the real valued function $h_i:\mathbb{R}^{ln}\rightarrow \mathbb{R}$ is Lipschitz continuous.
\end{remark}

The dynamics \eqref{ajl} addresses two questions, viz. how far from each other the local primary variables are and how fast they can be synchronized to a common trajectory. The quadratic appearance of $e_{iq}$ and $\dot{e}_{iq}$ in \eqref{ajl} ensures that it is monotonically increasing in $\mathbb{R}$.
\begin{figure}[t]
	\centering
	\includegraphics[width=4in]{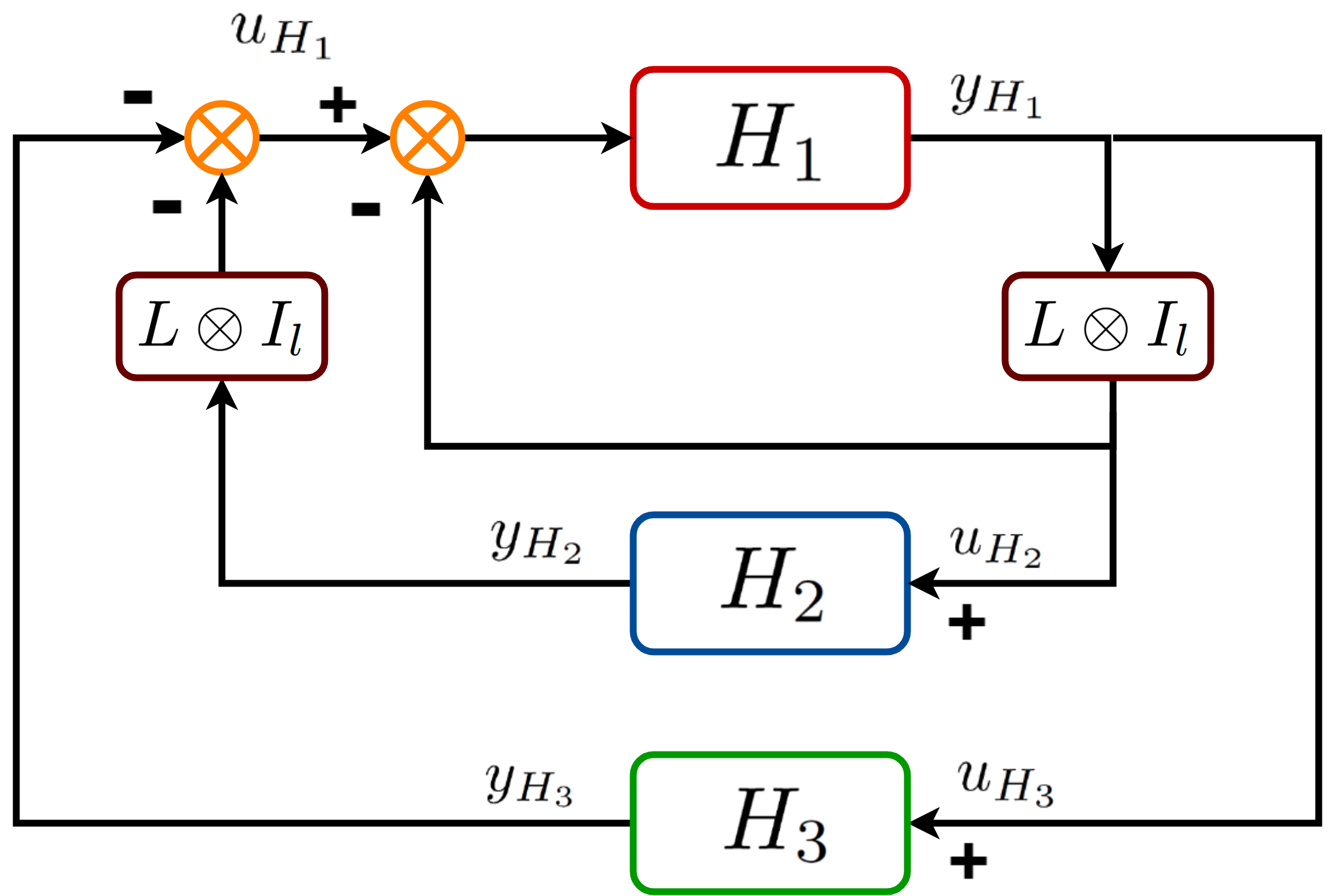}
	\caption{Interconnected networked dynamics of $H_1, H_2$ and $H_3$.}
	\label{dpd}
\end{figure}

\subsubsection{Integrating the adaptive coupling law \eqref{ajl} with the primal-dual dynamics \eqref{pddy}}\label{adpdd}
By integrating the adaptive coupling law \eqref{ajl} with the PDD \eqref{pddy} and partitioning the resulting dynamics into three interconnected subsystems i.e., $H_1$ (primal partition), $H_2$ (consensus dual partition), and $H_3$ (inequality dual partition) as shown in Fig. \ref{dpd}, yields:
\begin{align} 
H_1&:\begin{cases}
\dot{x}&=-\nabla_xf(x)+u_x+u_{H_1},\\
\dot{a}_{iq}&=d_{iq}(e^T_{iq}e_{iq}+\dot{e}^T_{iq}\dot{e}_{iq}),\forall i \in \mathcal{N},\forall q \in \mathcal{N}_i,\\
y_{H_1}&=x.
\end{cases}\label{h1d}\\
H_2&:\begin{cases}
\dot{\alpha}&=u_{H_2},\\
y_{H_2}&=\alpha.
\end{cases}\label{h2d}
\end{align}

The system $H_3$ represents the $\theta^{ik}_j$ dynamics in the stacked vector form with $u_{H_3}$ and $y_{H_3}$ as its input and output respectively, as given below:
\begin{align} 
H_3&:\begin{cases}
\dot{\theta}^{ik}_j&=[g_j(x_{ik})]^+_{\theta^{ik}_j},~\forall^l_{k=1};\forall^{m^{ik}_g}_{j=1}; \forall i \in N,\\
y_{H_3} &=\sum_{j=1}^{m^{ik}_g}\theta^{ik}_j\nabla_{x_{ik}}g_j(x_{ik}),\forall^l_{k=1};\forall^{m^{ik}_g}_{j=1}; \\&~~~\forall i \in N,
\end{cases}\label{h3}
\end{align}
where $y_{H_1},y_{H_2},y_{H_3} \in \mathbb{R}^{ln}$ and $u_{H_1}=-(L \otimes I_l)y_{H_2}-y_{H_3},u_{H_2}=(L\otimes I_{l})y_{H_1}$, and $u_{H_3}=y_{H_1}$. 

The ADPDD \eqref{h1d}-\eqref{h3} has been characterized as the feedback interconnected networked system as shown in the Fig. \ref{dpd}. Each agent in the underlying network is diffusively coupled with its neighboring agents by virtue of the communication topology that defines the interaction between such agents on the graph $\mathcal{G}(\mathcal{N},\mathcal{E})$. It can be noted that the network representation in Fig. \ref{dpd} is independent of the graph parameters such as communication topology, number of agents, and interaction links. Irrespective of such parameters, if the graph $\mathcal{G}(\mathcal{N},\mathcal{E})$ is connected, one can arrive at the stability results of the underlying network by only verifying its passivity properties. Towards this end, the following subsection first motivates the passivity analysis of the network shown in Fig. \ref{dpd} which further leads to its closed loop stability and robustness analysis. 

\subsection{Passivity based stability analysis of ADPDD}\label{passive_pd}
This section begins with passivity analysis of the subsystems $H_1$, $H_2$, $H_3$ and their feedback interconnection shown in Fig. \ref{dpd} and then moves towards the stability and robustness analysis of the said feedback interconnection. The Krasovskii type storage function has been defined for each subsystem (see, \cite{feijer2010stability}) which has led to a new passivity property with differentiation at both ports\cite[Proposition 2]{kosaraju2017control}. The intuition behind this proposition is to define the Krasovskii type storage function $V(x)$ for the dynamical system defined in \eqref{dynamic}, such that $\dot{V}\leq \dot{u}^T\dot{y}$,
where $\dot{u}$ and $\dot{y}$ are considered as port variables. This inequality shows that the map from the port input $\dot{u}$ to the port output $\dot{y}$ is passive. Motivated by this result, subsequently it is shown that the ADPDD is a passive system.
\subsubsection{$H_1$ is passive}
\begin{proposition}
	Assuming that the graph $G$ is connected and $f$ is strictly convex in $x$, if there exists $x_{eq}\in \mathbb{R}^{ln}$ that satisfies \eqref{dkkt}, then the subsystem $H_1$ is passive with port variables $(\dot{{y}}_{H_1},\dot{u}_{H_1})$. 
\end{proposition}
\begin{proof}
	Let $\tilde{a}_{iq} = a_{iq}-a^*_{iq}$ with $a^*_{iq}>0$ defined as follows:
	\begin{align} 
	a^*_{iq}=\begin{cases}
	a^*~~\textrm{if $i$ and $q$ are neighbors in $\mathcal{G}$},\\
	0~~\textrm{if $i$ and $q$ are not neighbors in $\mathcal{G}$},
	\end{cases}\label{a*}
	\end{align}
	where $a^*$ is a parameter to be selected.
	
	Consider the following storage function for the update law \eqref{ajl} \cite{shafi2015adaptive}.
	\begin{align} 
	W=\frac{1}{2}\sum_{i=1}^{p}\sum_{q=1}^{p}\frac{1}{d_{iq}}\tilde{a}^2_{iq}. \label{va} 
	\end{align}
	Differentiating \eqref{va} with respect to time yields the following:
	\begin{align} 
	\dot{W}=\sum_{i=1}^{p}\sum_{q=1}^{p}\tilde{a}_{iq}(e^T_{iq}e_{iq}+\dot{e}^T_{iq}\dot{e}_{iq}).\label{dva}
	\end{align}
	Acknowledging the graph symmetry and substituting for $e_{iq}=x_i-x_q$, \eqref{dva} modifies to
	\begin{align} 
	\begin{aligned}
	\dot{W} &= \dot{x}^T(L \otimes I_{l})\dot{x}-a^*\dot{x}^T(L \otimes I_{l})\dot{x}\\
	&~~~+x^T(L \otimes I_{l})x-a^*x^T(L \otimes I_{l})x,\\
	&= (1-a^*)\dot{x}^T(L \otimes I_{l})\dot{x}+(1-a^*)x^T(L \otimes I_{l})x.
	\end{aligned}\label{dva1}
	\end{align}
	
	Now, consider the following storage function for $H_1$, which is a sum of Krasovskii-type storage function of $x$ and \eqref{va}:
	\begin{align} 
	\begin{aligned}
	V_{H_1}(x) &= \frac{1}{2}\dot{x}^T\dot{x}+W.
	\end{aligned}\label{vh1}
	\end{align}
	
	Differentiating \eqref{vh1} with respect to time and using \eqref{dva1} yields,
	\begin{align}
	\begin{aligned}  
	\dot{V}_{H_1}(x)
	&=-\dot{x}^T\mathbb{H}\dot{x}-\dot{x}^T(L \otimes I_{l})\dot{x}+(1-a^*)\dot{x}^T(L \otimes I_{l})\dot{x}\\&~~~+(1-a^*){x}^T(L \otimes I_{l}){x}+\dot{x}^T\dot{u}_{H_1},\\
	&=-\dot{x}^T\mathbb{H}\dot{x}-a^*\dot{x}^T(L \otimes I_{l})\dot{x}+(1-a^*){x}^T(L \otimes I_{l}){x}\\&~~~+\dot{x}^T\dot{u}_{H_1},\\
	&\leq -\big(\lambda_{min}({\mathbb{H}})+a^*\lambda_{2}(L \otimes I_{l})\big)\norm{\dot{y}_{H_1}}^2\\&~~~+(1-a^*)\lambda_{2}(L \otimes I_{l})\norm{y_{H_1}}^2+\dot{{y}}_{H_1}^T\dot{u}_{H_1}.
	\end{aligned}\label{ddh1}
	\end{align}
	
	Notice that $y_{H_1}=x$ and choosing $a^*>1$ makes the term $(1-a^*)\lambda_{2}(L \otimes I_{l})\norm{y_{H_1}}^2$ in \eqref{ddh1} negative definite. Since $\lambda_{min}({\mathbb{H}})+a^*\lambda_{2}(L \otimes I_{l})>0$ for a non-negative value of $a^*$, the inequality \eqref{ddh1} implies that the subsystem $H_1$ is output strictly passive (\quotes{OSP}\cite{van2000l2}) with respect to the port variables $\dot{u}_{H_1}$ and $\dot{y}_{H_1}$.
\end{proof}

\subsubsection{$H_2$ is passive}
\begin{proposition}
	Assuming that the graph $G$ is connected and $f$ is strictly convex in $x$, if there exists $\alpha_{eq}\in \mathbb{R}^{ln}$ satisfying \eqref{dkkt}, then the subsystem $H_2$ is passive with port variables $(\dot{y}_{H_2},\dot{u}_{H_2})$.
\end{proposition}
\begin{proof}
	Consider a Krasovskii-type storage function for $H_2$ as given below:
	\begin{align} 
	{V}_{H_2}(\alpha) = \frac{1}{2}\dot{\alpha}^T\dot{\alpha}.\label{vh2}
	\end{align}
	Differentiating \eqref{vh2} with respect to time yields,
	\begin{eqnarray} 
	\dot{V}_{H_2}(\alpha)
	&=& \dot{\alpha}^T(L \otimes I_{l})\dot{x}, \nonumber\\
	&=&\dot{y}_{H_2}^T\dot{u}_{H_2}. \label{ddh21}
	\end{eqnarray}
	\eqref{ddh21} yields the following inequality,
	\begin{align} 
	V_{H_2}(\alpha(\tau))-V_{H_2}(\alpha(0)) \leq \int_{0}^{\tau} \dot{y}_{H_2}^T\dot{u}_{H_2}dt.\label{dpassh2}
	\end{align}
	Hence, the subsystem $H_2$ is passive with respect to port variables $\dot{u}_{H_2}$ and $\dot{y}_{H_2}$. 
\end{proof}
\subsubsection{$H_3$ is passive}
In the following, $H_3$ is modeled as a switched dynamical system.

The dynamics in \eqref{h3} becomes discontinuous when $\theta^{ik}_{j}=0$ and $ g_{j}(x_{ik})<0$. The value of $g_j(x_{ik})^+$ switches from $g_j(x_{ik})$ to $0$. To further clarify that, \eqref{h3} is reformulated below as given in Kose \cite{kose1956solutions}.
\begin{align} 
\dot{\theta}^{ik}_{j} =\begin{cases}
g_j(x_{ik}),~\mathrm{if}~\theta^{ik}_{j}>0~\mathrm{or}~ g_j(x_{ik})>0,\\
0. 
\end{cases}\label{proj}
\end{align}
From \eqref{proj}, the projection is seen to be active for the second case. Let $\mathcal{I}_i= \{1,\ldots,lm^{ik}_g\}$ and $\sigma_{i}:[0,\infty) \rightarrow \mathcal{I}_i,\forall k=1,\ldots,l; j \in \mathcal{I}_i$ be an arbitrary switching signal. Then
\begin{align} 
\sigma_{i}(t)=\{j|\theta^{ik}_{j}=0,g_j(x_{ik})\leq 0,\forall k;\forall j \in \mathcal{I}_i\},\label{switch}
\end{align} represents the switching time instances when there is an active projection.
Considering \eqref{switch}, the inequality constraint dynamics given in \eqref{h3} takes the form of a switched system:
\begin{align} 
\dot{\theta}^{ik}_{j}=\begin{cases}
g_j(x_{ik}),\forall k;j \notin \sigma_{i}(t),\\
0,\forall k;j \in \sigma_{i}(t),  
\end{cases}\label{thetadot}
\end{align}
where $\sigma_{i}(t) \subset \sigma(t), \forall^{n}_{i=1}$. Let $V_{H_3}$ be the Lyapunov function associated with $H_3$. It is defined as given below:
\begin{align} 
{V}_{H_3}(\theta) = \frac{1}{2}\sum_{i=1}^{n}\sum_{k=1}^{l}\sum_{j \notin \sigma_{i}(t)}^{}(\dot{\theta}^{ik}_j)^2.\label{vh3}
\end{align}
\begin{proposition}
	The subsystem $H_3$ is passive with port input $\dot{u}_{H_3}$, and port output $\dot{y}_{H_3}$ for each pair of switching time instances $(\tau^+_{\sigma_i},\tau^-_{\sigma_i})$ corresponding to \eqref{thetadot} where $\tau^-_{{\sigma_i}}<\tau^+_{{\sigma_i}}$ such that $\sigma_i(\tau^+_{\sigma_i})=\sigma(\tau^-_{\sigma_i})=\sigma_i \in \mathcal{I}_i$ and $\sigma_i(\tau') \neq \sigma_i$ for $\tau^-_{{\sigma_i}}<\tau<\tau^+_{\sigma_i}$.
\end{proposition}
\begin{proof}
	Differentiating \eqref{vh3} with respect to time yields, 
	\begin{align} 
	\dot{V}_{H_3}(\theta)&= \sum_{i=1}^{n}\sum_{k=1}^{l}\sum_{j\notin \sigma_{i}(t)}^{}\dot{\theta}^{ik}_j\ddot{\theta}^{ik}_j,\nonumber \\
	&=\sum_{i=1}^{n}\sum_{k=1}^{l}\sum_{j\notin \sigma_{i}(t)}^{}\dot{\theta}^{ik}_j\nabla g_j(x_{ik})\dot{x}_{ik}, \nonumber \\&\leq \sum_{i=1}^{n}\sum_{k=1}^{l}\sum_{j=1}^{m^{ik}_g}\dot{\theta}^{ik}_j\nabla g_j(x_{ik})\dot{x}_{ik},\nonumber\\
	&\leq \sum_{i=1}^{n}\sum_{k=1}^{l}\dot{x}_{ik}\sum_{j=1}^{m^{ik}_g}\dot{\theta}^{ik}_j\nabla g_j(x_{ik}). \label{dh31-1}
	\end{align}
	Using $u_{H_3}$ and $y_{H_3}$ from \eqref{h3} in \eqref{dh31-1},
	\begin{align} 
	\dot{V}_{H_3}(\theta)\leq\dot{y}_{H_3}^T\dot{u}_{H_3}.\label{dh31}
	\end{align}
	Thus,
	\begin{align} 
	\sum_{i=1}^{n}\sum_{k=1}^{l}\sum_{j \notin \sigma_i(t)}^{}V_k(\theta^{ik}_{j}(\tau^+_{\sigma_i}))-V_k(\theta^{ik}_{j}(\tau^-_{\sigma_i}))\leq\int_{\tau^-_{\sigma_i}}^{\tau^+_{\sigma_i}}\dot{y}_{H_3}^T\dot{u}_{H_3}dt. \label{passivity}
	\end{align} 
	\eqref{passivity} ensures that the switched system \eqref{thetadot} represents a finite family of passive systems. However, it must be ensured that the Lyapunov function $V_{H_3}$ does not increase during the switching events. In line with this, the following two cases have been considered:
	\begin{enumerate}
		\item It may happen for some $x_{ik}$ in \eqref{thetadot}, that the function $g_j(x_{ik})$ goes from negative to positive through $0$. This will cause the Lyapunov function to change from $V_k(\theta^{ik}_{j}(\tau^-_{\sigma_i}))$ to $V_k(\theta^{ik}_{j}(\tau^+_{\sigma_i}))$. If that happens, the Lagrangian multiplier $\theta^{ik}_{j}>0$ will add a new term to $V_k(\theta^{ik}_{j}(\tau_{\sigma_i}))$. Since, $V_k(\theta^{ik}_{j}(\tau_{\sigma_i}))$ is continuous in time, \eqref{passivity} holds for $\tau >\tau^-_{\sigma_i}$ as well as $\tau<\tau^+_{\sigma_i}$. Hence, 
		$V_k(\theta^{ik}_{j}(\tau^+_{\sigma_i}))=V_j(\theta^{ik}_{j}(\tau^-_{\sigma_i}))$.
		
		\item In this case the projection of $k^{th}$ constraint for a given $j$ becomes active, i.e., $\theta^{ik}_{j}$ reaches to $0$ from a positive value for the $k^{th}$ constraint of the $i^{th}$ machine. Hence, the corresponding $k^{th}$ term of the Lyapunov function $V_k(\theta^{ik}_{j})$ will disappear. In turn, the following inequality will be satisfied.	$V_k(\theta^{ik}_{j}(\tau^+_{\sigma_i}))<V_k(\theta^{ik}_{j}(\tau^-_{\sigma_i}))$.
	\end{enumerate}
	Hence, in both the cases, the Lyapunov function $V_k(\theta^{ik}_{j}(\tau))$ will be non-increasing. 
\end{proof}
\subsection{Stability analysis of the feedback interconnection shown in Fig \ref{dpd}.}\label{gb}	
\begin{proposition}\label{prof3.4}
	Let $a^*>1$ then the interconnected network dynamics \eqref{h1d}-\eqref{h3} is passive from the input $u_x$ to the output $y_{H_1}$.
\end{proposition}
\begin{proof}
	Let $V$ be the candidate Lyapunov function for the interconnected system represented in Fig. \ref{dpd} such that
	\begin{align} 
	V = V_{H_1}+V_{H_2}+V_{H_3}. \label{ly}
	\end{align}
	Differentiating \eqref{ly} and using \eqref{ddh1}, \eqref{ddh21}, \eqref{dh31} yields
	\begin{align} 
	\dot{V} &= \dot{V}_{H_1}+\dot{V}_{H_2}+\dot{V}_{H_3},\nonumber\\
	&\leq-\big(\lambda_{min}({\mathbb{H}})+a^*\lambda_{2}(L \otimes I_{l})\big)\norm{\dot{y}_{H_1}}^2\nonumber\\&~~~+(1-a^*)\lambda_{2}(L \otimes I_{l})\norm{y_{H_1}}^2+\dot{{y}}_{H_1}^T\dot{u}_{H_1}+\dot{u}_{H_2}^T\dot{y}_{H_2}\nonumber\\&~~~+\dot{u}_{H_3}^T\dot{y}_{H_3}\\
	&\leq -\big(\lambda_{min}({\mathbb{H}})+a^*\lambda_{2}(L \otimes I_{l})\big)\norm{\dot{y}_{H_1}}^2\nonumber\\&~~~+(1-a^*)\lambda_{2}(L \otimes I_{l})\norm{y_{H_1}}^2.\label{hwz}
	\end{align}
	Thus, the interconnected network dynamics \eqref{h1d}-\eqref{h3} is passive, if $a^*>1$ strictly holds. 
	
	The following result establishes the boundedness of the trajectories of \eqref{h1d}-\eqref{h3}.
	\begin{proposition}\label{bounded}
		The trajectories of \eqref{h1d}-\eqref{h3} are bounded for any finite initial conditions.
	\end{proposition}
	\begin{proof}
		To show that the trajectories of \eqref{h1d}-\eqref{h3} are bounded, consider the following storage function:
		\begin{align}
		\bar{V} = \frac{1}{2}\norm{x-x^*}^2+\frac{1}{2}\norm{\alpha-\alpha^*}^2+\frac{1}{2}\norm{\theta-\theta^*}^2+W.\label{vbar}
		\end{align}
		where $W$ is the storage function defined in \eqref{va}.
		Differentiating \eqref{vbar} with respect to time yields
		\begin{align}
		\dot{\bar{V}}&=-\nabla_x \bar{\mathcal{L}}(x,\alpha,\theta)^T(x-x^*)+\nabla_\alpha \bar{\mathcal{L}}(x,\alpha,\theta)^T(\alpha-\alpha^*)\nonumber\\
		&~~~+(\theta-\theta^*)^T\nabla_\theta [\bar{\mathcal{L}}(x,\alpha,\theta)]^+_\theta+\dot{W},\nonumber\\
		&\leq -\nabla_x \bar{\mathcal{L}}(x,\alpha,\theta)^T(x-x^*)+\nabla_\alpha \bar{\mathcal{L}}(x,\alpha,\theta)^T(\alpha-\alpha^*)\nonumber\\
		&~~~+\sum_i \sum_k \sum_j(\theta^{ik}-(\theta^{ik}_j)^*)[g_j(x_{ik})]^+_{\theta^{ik}_j}+\dot{W},\nonumber\\
		&\leq -\nabla_x \bar{\mathcal{L}}(x,\alpha,\theta)^T(x-x^*)+\nabla_\alpha \bar{\mathcal{L}}(x,\alpha,\theta)^T(\alpha-\alpha^*)\nonumber\\
		&~~~+(\theta-\theta^*)^Tg(x)+\dot{W}.\label{end is near}
		\end{align}
		Note that $(\theta^{ik}_j-(\theta^{ik}_j)^*)g_j(x_{ik})\geq 0,\forall j \in \sigma_{i}(t)$ because $g_j(x_{ik})<0$ and $\theta^{ik}_j=0$ as confirmed by \eqref{thetadot}. Using first order condition of convexity-concavity of the Lagrangian function \eqref{dlag} and replacing $\dot{W}$ by right-hand side of \eqref{dva1}, \eqref{vbra} modifies to the following:
		\begin{align}
		\dot{\bar{V}}
		&\leq
		-[\bar{\mathcal{L}}(x,\alpha,\theta)-\bar{\mathcal{L}}(x^*,\alpha,\theta)]+[\bar{\mathcal{L}}(x,\alpha,\theta)-\bar{\mathcal{L}}(x,\alpha^*,\theta)]\nonumber\\
		&~~~+[\bar{\mathcal{L}}(x,\alpha,\theta)-\bar{\mathcal{L}}(x,\alpha,\theta^*)]\nonumber\\
		&~~~+(1-a^*)\dot{x}^T(L \otimes I_{l})\dot{x}+(1-a^*)x^T(L \otimes I_{l})x.\label{vbra}
		\end{align}
		Since $(x^*,\alpha^*,\theta^*)$ is the saddle-point of \eqref{dlag}, with $a^*>1$ yields the following
		\begin{align}
		\dot{\bar{V}}\leq 0.\label{bounda}
		\end{align}
		which is sufficient to ensure that the trajectories of \eqref{h1d}-\eqref{h3} are bounded.
	\end{proof}
	
	In what follows, the asymptotic stability of the saddle-point solution of \eqref{h1d}-\eqref{h3} is established. To this end, the underlying networked dynamics is represented as a hybrid system wherein $H_1$, $H_2$ are represented as continuous-time dynamical systems and $H_3$ is represented as a system with right-hand side discontinuity. The framework of LaSalle's invariance principle for hybrid dynamical systems (see, \cite{lygeros2003dynamical}) is stated below, which in our case provides a useful result on the convergence of \eqref{h1d}-\eqref{h3} to the saddle point solution that satisfies \eqref{dkkt}.
	\begin{proposition}\label{hyb}
		Consider the hybrid networked dynamics \eqref{h1d}-\eqref{h3} and let $z = [x^T,\alpha^T,\theta^T]^T \in \mathcal{X} \subseteq \mathbb{R}^{ln(2+m^{ik}_g)}$, and $\Psi \subseteq \mathcal{X}$ be compact and positively invariant. Assuming that the Lyapunov function $V$ defined in \eqref{ly} is continuously differentiable and $\dot{V} \leq 0$ along the trajectories of $z(t) \in \Psi$, every trajectory in $\Psi$ converges to $\epsilon$, where $\epsilon \subset \Psi$ is a maximal positive invariant set of $\Psi$ such that
		\begin{enumerate} 
			\item $\dot{V}=0$ for a fixed $\sigma$.
			\item $V_k(\theta^{ik}_{j}(\tau^+_{\sigma_i}))=V_k(\theta^{ik}_{j}(\tau^-_{\sigma_i}))$ for a switching instance $\tau$ between $\tau^-_{\sigma_i}$ and $\tau^+_{\sigma_i}$.
		\end{enumerate}
	\end{proposition}
\end{proof}
Proposition \ref{hyb} gives the next result on the convergence of \eqref{h1d}-\eqref{h3} to the saddle point solution that satisfies the conditions in \eqref{dkkt}.
\begin{proposition}\label{thmn-1}
	The hybrid network dynamics \eqref{h1d}-\eqref{h3} converges to the saddle point solution $x^*,\alpha^*,\theta^*$ satisfying \eqref{dkkt}.
\end{proposition}
\begin{proof}
	From Proposition \ref{hyb}, for a fixed $\sigma$, $\dot{V}=0$. Thus the primal as well as dual dynamics in \eqref{h1d}-\eqref{h3} converge to the saddle point solution  contained within the set $\epsilon$. If $g_j(x^*_{ik})<0$ then $(\theta^{ik}_{j})^*=0$. However, if $g_j(x^*_{ik})>0$, then $(\theta^{ik}_{j})^*$ will penalize the constraint violation by rising to a large value. Since all trajectories are bounded, it contradicts the continuity of $V$, thus $\dot{\theta}^{ik}_{j}=0$. To this end, the solutions of \eqref{h1d}-\eqref{h3} also satisfy the KKT conditions \eqref{dkkt} and yield the saddle point solution $(x^*,\alpha^*,\theta^*)$.
\end{proof}
Choosing $a^*>1$ and using \eqref{ux}, \eqref{hwz} modifies to 
\begin{align}
\dot{V}&\leq-\big(\lambda_{min}({\mathbb{H}})+a^*\lambda_{2}(L \otimes I_{l})\big)\norm{\dot{y}_{H_1}}^2 \nonumber\\&~~~+(1-a^*)\lambda_{2}(L \otimes I_n)\norm{y_{H_1}}^2\leq 0.\label{ndv1}
\end{align} 
\begin{proposition}\label{thmn}
	The saddle point solution of \eqref{dlag} is asymptotically stable. 
\end{proposition}
\begin{proof}
	The proof is straightforward from Proposition \ref{prof3.4} and Proposition \ref{thmn-1} and \eqref{ndv1}.
\end{proof}
In the recent article \cite{cherukuri2018role} the global asymptotic stability of the primal-dual dynamics is proved by using the Lyapunov function similar to  that of the sum of Krasovskii-type Lyapunov function \eqref{ly} and the Lyapunov function defined in \eqref{vbar}. This result can be extended to the globally asymptotic stability of the saddle-point of \eqref{dlag}.
\begin{remark}\label{remarkcherku}
	Let $\tilde{V}:\mathbb{R}^{ln} \times \mathbb{R}^{ln} \times \mathbb{R}^{lnm^{ik}_g} \rightarrow \mathbb{R}$ denote the candidate Lyapunov function for the ADPDD \eqref{h1d}-\eqref{h3}, given as sum of the candidate Lyapunov functions \eqref{ly} and \eqref{vbar} as follows:
	\begin{align}
	\tilde{V} = V+\bar{V}.\label{tildev}
	\end{align} If Assumption \ref{asss} holds then the trajectories of \eqref{h1d}-\eqref{h3} converge to the saddle-point $(x^*,\alpha^*,\theta^*)$ which is globally asymptotically stable. The proof of the Remark would be similar to proof the of \cite[Theorem 5.1]{cherukuri2018role}. Hence it is omitted from here to avoid repetition. 
\end{remark}

With the global asymptotic stability of the proposed dynamics \eqref{h1d}-\eqref{h3} established, the subsequent section addresses its rate of convergence and its comparison with the rate of convergence with the primal-dual dynamics without adaptive weights. 
\subsection{Accelerated convergence using ADPDD}\label{fasterc}
Let $\mathcal{A}\subseteq\mathbb{R}^{|\mathcal{E}|}_{>0}$ define the set of coupling weights, and $|\mathcal{E}|$ define the cardinality of the edge set $\mathcal{E}$. Then, in view of its definition, the Laplacian matrix $L\otimes I_l$ is a parameter varying, real and symmetric matrix, which is differentiable and uniformly continuous on $\mathcal{A}$. As a consequence, the following hold:
\begin{statement}
	There exists $\Lambda>0$ such that the spectral norm $\norm{L \otimes I_l}<\Lambda,\forall a_{iq} \in \mathcal{A},\forall q \in \mathcal{N}_i,\forall i \in \mathcal{N}$. 
\end{statement}
\begin{statement}\label{as2}
	The gradient of $L \otimes I_l$ with respect to $a_{iq}$ is bounded above by some scalar $\eta$,
	$\left\| \nabla L \otimes I_l \right\| \leq \eta, a_{iq}\in \mathcal{A}$. 
\end{statement}

Let $L_0\otimes I_l$ be the Laplacian matrix of $\mathcal{G}$ whose coupling weights are constant parameters, then $L_0\otimes I_l$ results in a constant matrix.
\begin{proposition}\label{lemma12}
	If the coupling weights evolve according to the law \eqref{ajl}, then the following holds $\forall t>t_0$ :
	\begin{align}
	\lambda_{2}(L \otimes I_l) > \lambda_{2}(L_0\otimes I_l).\label{fast}
	\end{align}
\end{proposition}
\begin{proof}
	If $E$ is the incidence matrix of the undirected graph $\mathcal{G}$, then the Laplacian matrices $L_0$ and $L \otimes I_l$ can be written as:
	\begin{align}
	L_0\otimes I_l &= EE^T \label{eet}\\
	L \otimes I_l &= EC(t)E^T\label{ecet}
	\end{align}
	where $C(t)$ is a diagonal matrix containing the coupling weights. To prove \eqref{fast}, it is first proved that $x^TEC(t)E^Tx\geq x^TEE^Tx$.
	\begin{align}
	x^TEC(t)E^Tx-x^TEE^Tx&=x^T(EC(t)E^T-EE^T)x\nonumber\\
	&=x^T[E(C(t)-I)E^T]x.
	\end{align}
	For an undirected graph $\mathcal{G}$, $a_{iq}(t_0)\geq 1,\forall q \in \mathcal{N}_i,\forall i \in \mathcal{N}$. Then $\forall(q,i) \in \mathcal{E}$, $C(t)\geq I_{|\mathcal{E}|}$. Hence, 
	\begin{align}
	C(t)-I\geq 0, \label{cgi}
	\end{align} in fact, $C(t_0)$ is a diagonal matrix with the coupling weights $a_{iq}(t_0)$, thus $C(t)\geq C(t_0),\forall t>t_0$. Thus from the above reasoning, and \eqref{cgi},
	\begin{align}
	x^TEC(t)E^T\geq x^TEE^Tx\nonumber
	\end{align}
	From \eqref{eet} and \eqref{ecet},
	\begin{align}
	x^TL\otimes I_lx\geq x^TL_0\otimes I_lx. \label{proved}
	\end{align}
	Let $\lambda_i$ be the $i^{th}$ eigenvalue in the ordered-pair of eigenvalues represented below:
	\begin{align}
	\lambda_{2} \leq \ldots \leq \lambda_{i} \leq\ldots\leq\lambda_n.\label{inequ}
	\end{align}
	Then  according to Courant-Fischer theorem \cite{horn1990matrix},
	\begin{align}
	\lambda_i(EE^T) &= \min_{x\neq 0,x\perp v_1}\frac{x^TEE^Tx}{x^Tx}\nonumber\\
	&\leq\min_{x\neq 0,x\perp v_1}\frac{x^TEC(t)E^Tx}{x^Tx}=\lambda_i(EC(t)E^T)
	\end{align} where $v_1$ is the eigenvector (vector of all ones) corresponding to the eigenvalue $\lambda_1=0$.
	Thus for $i=2$,
	\begin{align}
	\lambda_{2}(EE^T) \leq \lambda_{2}(EC(t)E^T)\nonumber\\
	\lambda_{2}(L_0\otimes I_l) \leq \lambda_{2}(L\otimes I_l)\nonumber\\
	\lambda_{2}(L\otimes I_l) > \lambda_{2}(L_0\otimes I_l),\forall t>t_0.\label{l1}
	\end{align}
\end{proof}
\begin{proposition}\label{lemma23}
	If the coupling weights evolve according to \eqref{ajl}, then the following always hold:
	\begin{align}
	\lambda_{2}(L \otimes I_l) \leq \frac{\lambda_n(L \otimes I_l)}{\lambda_n (L_0 \otimes I_l)}\lambda_2(L_0 \otimes I_l).
	\end{align} 
\end{proposition}
\begin{proof}
	The proof simply follows from the inequality \eqref{inequ}. Taking the ratio of the ordered pair of eigenvalues of $L \otimes I_l$ and $L_0 \otimes I_l$, yields the following:
	\begin{align}
	\frac{\lambda_{2}(L \otimes I_l)}{\lambda_{2}(L_0 \otimes I_l)}\leq \frac{\lambda_{n}(L \otimes I_l)}{\lambda_{n}(L_0 \otimes I_l)},\label{inequ2}
	\end{align}
	But, for $t>t_0$, the inequality \eqref{inequ2} strictly holds. Thus 
	\begin{align}
	\lambda_{2}(L \otimes I_l) \leq \frac{\lambda_n(L \otimes I_l)}{\lambda_n (L_0 \otimes I_l)}\lambda_2(L_0 \otimes I_l). \label{lambdamax}
	\end{align} 
\end{proof}
Proposition \ref{lemma12} and \ref{lemma23} can be further used to prove that the adaptive primal dual dynamics has an accelerated, yet bounded convergence rate as compared to the conventional primal-dual dynamics. 
\begin{proposition}\label{prope}
	If the inequality \eqref{l1} holds, then the primal dynamic in \eqref{h1d}, under the adaptive coupling law \eqref{ajl}, achieves accelerated convergence.
\end{proposition}
\begin{proof}
	Below a timescale separation is enforced in the dynamics of the primal subsystem $H_1$,
	\begin{align} 
	\dot{x}&=-\nabla_xf(x)+u_x+u_{H_1},\\
	\dot{a}_{iq}&=\epsilon d_{iq}(e^T_{iq}e_{iq}+\dot{e}^T_{iq}\dot{e}_{iq}),\forall i \in \mathcal{N},\forall q \in \mathcal{N}_i,
	\end{align} 
	with $\epsilon<<1$ ensuring that the primal variable $x_i$ evolves faster than the coupling weights $a_{iq}$.
	
	The primal subsystem has two control inputs $u_x$, to study the primal dynamics with respect to $u_x$ in \eqref{ux}, let us analyze the primal subsystem $H_1$ when $u_{H_1}$ is at steady state or equal to $0$. With the assumption that the coupling weight dynamics is much slower, the primal dynamics is re-written as:
	\begin{align} 
	\dot{x}&=-\nabla_xf(x)-(L \otimes I_l)x+u_{H_1},\nonumber\\
	&=-F(x)+u_{H_1},\label{new primal}
	\end{align} where $F(x) = \nabla_xf(x)+(L \otimes I_l)x$ .
	
	Using Assumption \ref{asss} it can be proved that the primal dynamics \eqref{new primal} is strongly monotone for all $x \in \mathbb{R}^{ln}$ by evaluating the Jacobian of $F(x)$, i.e. $\nabla F(x) = \mathbb{H} + L \otimes I_l \geq \mu I$, where $\mu$ is the modulus of convexity of $f$ (from Assumption \ref{asss}. Since $\mu>0$, the Jacobian $\nabla F(x)$ is symmetric and positive definite $\forall x \in \mathbb{R}^{ln}$. It proves that $F(x)$ is strongly monotone by virtue of which the primal dynamics \eqref{new primal} converges to the unique optimizer $x^*$. Thus uniqueness of the primal optimizer remains invariant under the adaptive coupling law \eqref{ajl}. 
	
	The following result establishes the accelerated convergence of \eqref{new primal} with respect to the unique optimizer $x^*$.
	Let $V_{H_1}$ define the Lyapunov candidate function as given below:
	\begin{align}
	V_{H_1}=\frac{1}{2}\dot{x}^T\dot{x}.
	\end{align}
	Differentiating $V_{H_1}$ with respect to time $t$,
	\begin{align}
	\dot{V}_{H_1}&\leq -(\lambda_{min}(\mathbb{H})+\lambda_{2}(L \otimes I_l))\dot{x}^T\dot{x}\\
	&\leq -\lambda_mV_{H_1} \label{finalvdot}
	\end{align}
	where $\lambda_m = 2(\lambda_{min}(\mathbb{H})+\lambda_{2}(L \otimes I_l))$.
	Therefor, 
	\begin{align}
	V_{H_1}(x(t))\leq V_{H_1}(x(t_0))\exp\{-\lambda_mt\},\forall t\geq t_0.
	\end{align}
	or
	\begin{align}
	\norm{x-x^*}\leq \sqrt{2V_{H_1}(x(t_0))}\exp\{-0.5\lambda_mt\},\forall t\geq t_0.\label{globe}
	\end{align}
	Further, since the primal-dual dynamics has a bounded convergence with respect to the saddle point solution (see Proposition \ref{bounda}), using  Assumption \ref{lipsch}, and Remark \ref{remarkcherku}, every initial condition $x(t_0)\in \mathbb{R}^{ln}$ approaches the optimal solution $x^*$ faster than the usual. Thus the accelerated convergence holds globally.
	Considering the upper bound on $\lambda_{2}(L \otimes I_l)$ as given in \eqref{lambdamax}, let $\lambda_{2}(L \otimes I_l)=\frac{\lambda_n(L \otimes I_l)}{\lambda_n (L_0 \otimes I_l)}\lambda_2(L_0 \otimes I_l)$ and $\lambda_{m_0}= 2(\lambda_{min}(\mathbb{H})+\lambda_{2}(L_0 \otimes I_l))$. Then it is seen that $\lambda_m=\lambda_{min}(\mathbb{H})+\frac{\lambda_n(L \otimes I_l)}{\lambda_n (L_0 \otimes I_l)}\lambda_2(L_0 \otimes I_l)>>\lambda_{m_0}$.
	Hence proved.
\end{proof}
\begin{remark}
	The ADPDD \eqref{h1d}-\eqref{h3} achieves accelerated convergence to the saddle point solution that satisfies the KKT conditions \eqref{kkta}. 
\end{remark}
\begin{proof}
	The proof follows from Proposition \ref{thmn-1} and Proposition \ref{prope}. The occurrence of the primal optimizer $x^*$ and the dual optimizers is simultaneous. 
	
	Recall from \eqref{h1d}-\eqref{h3} that $y_{H_1}=x, y_{H_2}=\alpha$, and $y_{H_3} =\sum_{j=1}^{m^{ik}_g}\theta^{ik}_j\nabla_{x_{ik}}g_j(x_{ik}),\forall^l_{k=1};\forall^{m^{ik}_g}_{j=1}$. $y_{H_1}=x^* \implies \dot{y}_{H_1}=0$ which also implies that $\dot{\alpha}=0$ and $\dot{\theta}=0$. 
	Thus accelerated convergence of \eqref{h1d} to the primal optimizer $x^*$ implies the accelerated convergence of both \eqref{h2d} and \eqref{h3} converge to the dual optimizers $(\alpha^*,\theta^*)$.
\end{proof}
\begin{remark}\label{pdot}
	Since the adaptive synchronization in the primal variables results in an accelerated convergence to the primal optimizer $x^*$, the synchronization error $x_i-x_j,\forall j\in \mathcal{N}_i,\forall i\in \mathcal{N}$ remains lower than that of DPDD for all time. This significantly reduces the consensus constraint violation and thus the dual variable $\alpha$ pertaining to the consensus constraints does not become unnecessarily large for the ADPDD problem. The uniqueness of the primal optimizer for both Lagrangian functions \eqref{dlag} and \eqref{dlag1} owes to the strongly monotone property of $F(x)$ as discussed in the Proposition \ref{prope}. However, the dual dynamics \eqref{h2d} (which solely a function of the adaptively coupled primal variables) is not strongly monotone, which implies that the Lagrangian \eqref{dlag} and \eqref{dlag1} are not strongly concave with respect to $\alpha$. Thus, the dual dynamics \eqref{h2d} under the effect of adaptively synchronized primal variables and the one without adaptively synchronized primal variables settle to different equilibrium states. This further indicates that there exists a unique primal optimizer for both Lagrangian functions \eqref{dlag} and \eqref{dlag1} but the same does not hold for the dual optimizer $\alpha^*$. Since the dual variable $\theta^*$ pertains to the local inequality constraints, it remains unaffected by the adaptive synchronization in primal variables. Hence, the dual optimizer $\theta^*$ is also unique for both \eqref{dlag} and \eqref{dlag1}.
\end{remark}

The convergence rate of the distributed primal-dual dynamics is improved under the influence of adaptive synchronization. However, it may adversely affect the robustness of the proposed dynamics. Thus, there arises a necessity to quantify the robustness of the proposed dynamics with respect to the rate of convergence. The analysis presented below obtains a relation between the convergence rate of the proposed dynamics and its $L_2$-gain.
\subsection{Robustness analysis of the network dynamics with respect to the exogenous inputs}\label{l2stab}
Before proceeding with the robustness analysis of this section, it is worth noting the following remark on robustness property of the passive dynamical systems.
\begin{remark}\label{arjan}
	From the inequalities \eqref{ddh1}, \eqref{dpassh2}, and \eqref{passivity}, it is apparent that the interconnected network dynamics comprising \eqref{h1d}-\eqref{h3} is passive, and inherently robust to the perturbations arising in the primal and dual variables [see, Proposition 4.3.1, Remark 4.3.3 of \cite{van2000l2}].  
\end{remark}

Remark \ref{arjan} states the qualitative behavior of the proposed dynamics with respect to the notion of robustness. In the following, the robustness of the proposed dynamics against exogenous inputs is quantified in terms of the $L_2$-gain.

Consider without loss of generality, the new inputs to \eqref{h1d}-\eqref{h3} as
\begin{align} 
\begin{aligned}
H_1:\tilde{u}_{H_1}=u_{H_1}+\bigtriangleup u_{H_1},\\H_2:\tilde{u}_{H_2}=u_{H_2}+\bigtriangleup u_{H_2},\\
H_3:\tilde{u}_{H_3}=u_{H_3}+\bigtriangleup u_{H_3},\label{ni}
\end{aligned}
\end{align} respectively, where $\bigtriangleup u_{(.)}$ corresponds to the perturbations in the input $u_{(.)}\in \mathbb{R}^{ln}$. As discussed in \cite{kosaraju2018primal}, $\bigtriangleup u_{(.)}$ represent additive uncertainties or disturbances such as the numerical error accumulated in the corresponding variables. In what follows, the robustness of the ADPDD is quantified using $L_2$-gain analysis of dynamical systems.
Let ${\tilde{u}}=[{\tilde{u}}^T_{H_1},{\tilde{u}}^T_{H_2},{\tilde{u}}^T_{H_3}]^T$ and ${y}=[{y}^T_{H_1},{y}^T_{H_2},{y}^T_{H_3}]^T$.
\begin{proposition}\label{l21}
	The interconnected network dynamics \eqref{h1d}, \eqref{h2d}, and \eqref{h3} with $a_{iq}$ updated according to \eqref{ajl}, remains $L_2$ stable with the $L_2$-gain, $\gamma\leq \frac{1}{\lambda_{min}({\mathbb{H}})+a^*\lambda_{2}(L \otimes I_{l})}$, if $a^*>1$ strictly holds.
\end{proposition}
\begin{proof}
	Replacing the inputs in \eqref{h1d}-\eqref{h3} by the new ones as defined in \eqref{ni}, the time differential of the Lyapunov function \eqref{ly} modifies to the following:
	\begin{align} 
	\dot{V}&\leq-\big(\lambda_{min}({\mathbb{H}})+a^*\lambda_{2}(L \otimes I_{l})\big)\norm{\dot{y}_{H_1}}^2 \nonumber\\&~~~+(1-a^*)\lambda_{2}(L \otimes I_n)\norm{y_{H_1}}^2+\dot{{y}}^T_{H_1}\dot{\tilde{u}}_{H_1}\nonumber\\&~~~+\dot{y}_{H_2}^T\dot{\tilde{u}}_{H_2}+\dot{y}_{H_3}^T\dot{\tilde{u}}_{H_3}. \label{intra}\end{align}
	Acknowledging that $y_{H_1}=x$ and using \eqref{ux} in \eqref{intra} further yields
	\begin{align}
	\dot{V}&\leq-\big(\lambda_{min}({\mathbb{H}})+a^*\lambda_{2}(L \otimes I_{l})\big)\norm{\dot{y}_{H_1}}^2 \nonumber\\&~~~-(1-a^*)y_{H_1}^Tu_x+\dot{\tilde{u}}^T\dot{y},\label{ndv11}
	\end{align}
	where $\lambda_{min}({\mathbb{H}})+a^*\lambda_{2}(L \otimes I_{l})>0$ since $\mathbb{H}$ is positive definite. With $a^*>1$, the $L_2$-gain of the interconnected network dynamics, from the port input $\dot{\tilde{u}}$ to the port output $\dot{y}$ can be calculated by setting $u_x$ to $0$. 
	From inequality \eqref{ndv11}, the map from the input $\dot{\tilde{u}}$ to the output $\dot{y}$ remains finite $L_2$-gain stable around the saddle point $x^*,\alpha^*,\theta^*$, when the corresponding $L_2$-gain, satisfies
	\begin{align}
	\gamma\leq \frac{1}{\big(\lambda_{min}({\mathbb{H}})+a^*\lambda_{2}(L \otimes I_{l})\big)}. \label{l2bound}    
	\end{align}
\end{proof}
The inequality \eqref{l2bound}, clearly indicates that the $L_2$ gain corresponding to the adaptive distributed primal-dual dynamics reduces in margin as compared to the $L_2$ gain corresponding to the distributed primal-dual dynamics (without adaptive synchronization). Using \eqref{lambdamax}, one can obtain the following expression for the $L_2$-gain in the worst case:
\begin{align}
\underline{\gamma} = \frac{1}{\big(\lambda_{min}({\mathbb{H}})+a^*\frac{\lambda_n(L \otimes I_l)}{\lambda_n (L_0 \otimes I_l)}\lambda_2(L_0 \otimes I_l)\big)}. \label{worst-gain}
\end{align}
Comparing \eqref{l2bound} and \eqref{worst-gain}, it can be found out that the $L_2$-gain for the ADPDD has a reduced margin than that of the DPDD. Thus the algorithm calls for trade-off between the robustness and the accelerated convergence of the proposed dynamics. While the adaptive synchronization improves the rate of convergence of the primal-dual dynamics, it simultaneously degrades the robustness of the proposed algorithm wherein the worst-case $L_2$-gain is quantified by $\underline{\gamma} (<\gamma)$ in \eqref{worst-gain}.
\section{Applications and Numerical Examples}\label{sims}
This section discusses the application of the proposed dynamics to the distributed optimization problems concerning least squares\cite{nedic2018distributed,mead2010least} and support vector machines\cite{cortes1995support}. These problems are solved online over a network of wireless sensors or computing devices, in such premises the rate of convergence is a vital factor. In the following, the proposed dynamics \eqref{h1d}-\eqref{h3} is employed to solve the distributed least squares\cite{nabavi2015distributed} and distributed support vector machines\cite{forero2010consensus,stolpe2016distributed} problems.
\subsection{Distributed Least Squares}
Distributed least squares problems have been widely studied over recent years\cite{sayed2006distributed,mateos2009distributed,wang2019distributed}. These techniques have found applications in parameter estimation over wireless sensor networks \cite{kar2012distributed}, estimation of electro-mechanical oscillation modes of large power system networks \cite{nabavi2015distributed,zhao2017supervisory} etc. Each agent in the network is given a task to simultaneously and iteratively compute the same least squares solution to the linear equation $Ax=b$ where $A\in \mathbb{R}^{r_1 \times r_2}$  with $r_1>r_2$ and $b \in \mathbb{R}^{r_1 \times 1}$. 

Formally, the least squares problem is defined as given below\cite{bjorck1996numerical}:
\begin{align}
\min_x \frac{1}{2}\norm{Ax-b}^2.\label{lsq}
\end{align}
\subsubsection{Data partitioning}
It is assumed that each agent in the network adheres to $n_r = r_1/n$ consecutive rows of $A$ and $b$. For the sake of simplicity, equal partitioning of the rows of $A$ is considered. However, the proposed approach would hold even if the partitioning is uneven. 
\begin{align}
A=\begin{bmatrix}
A_1\\
A_2\\
\vdots\\
A_n
\end{bmatrix},	b=\begin{bmatrix}
b_1\\
b_2\\
\vdots\\
b_n
\end{bmatrix}.
\end{align} where $A_i\in\mathbb{R}^{n_r \times {l}}$ and $A_i\in\mathbb{R}^{n_r \times 1}$.
\subsubsection{Distributed formulation of least squares problem}
The consensus-based distributed optimization formulation of \eqref{lsq} would require the local estimates $x_1,x_2,\ldots,x_n$ to reach consensus on the global optimizer $x^*$. With data partitioning as defined above, the distributed version of the least squares problem \eqref{lsq}\cite{nabavi2015distributed} is defined as
\begin{align}
&\min_x\sum_{i=1}^{n}\frac{1}{2}\norm{A_ix-b_i}^2\nonumber\\
&\mathrm{subject~to}~ x_i = x_j,\forall j\in \mathcal{N}_i. \label{dlsq}
\end{align}
\subsubsection{Solution to the distributed least squares problem \eqref{dlsq} using ADPDD}
The Lagrangian problem corresponding to \eqref{dlsq} can be defined as
\begin{align}
\mathcal{L}(x,\alpha)=\sum_{i=1}^{n}\norm{A_ix-b_i}^2+\alpha^TL\otimes I_l x+x^T L\otimes I_lx. \label{lagdlsq}
\end{align}
Similarly to \eqref{pddy}, the proposed dynamics can be derived from \eqref{lagdlsq} as given below:
\begin{align} 
H_1&:\begin{cases}
\dot{x}&=-A^T(Ax-b)-(L \otimes I_l)x+u_{H_1},\\
\dot{a}_{iq}&=d_{iq}(e^T_{iq}e_{iq}+\dot{e}^T_{iq}\dot{e}_{iq}),\forall i \in \mathcal{N},\forall q \in \mathcal{N}_i\\
y_{H_1}&=x.
\end{cases}\label{p1lsq}\\
H_2&:\begin{cases}
\dot{\alpha}&=u_{H_2},\\
y_{H_2}&=\alpha.
\end{cases}\label{p2lsq}
\end{align} where $u_{H_1}=-(L \otimes I_l) y_{H_2}$ and $u_{H_2}=(L \otimes I_l)y_{H_1}$.
\subsubsection{Simulations}
The simulation parameters are randomly generated matrix $A\in \mathbb{R}^{100 \times 80}$ and vector $b\in \mathbb{R}^{100 \times 1}$. The network with a cyclic graph topology is assumed to comprise of $\mathrm{4}$ agents wherein each agent holds $A_i\in \mathbb{R}^{25\times 80}$ component of $A$ as well as the respective $b_i$. Each agent in the network computes $x\in \mathbb{R}^{80}$ local estimates and reaches consensus over the global solution $x^*$ as shown in the Fig. \ref{f21}. The simulations were carried out using $d_{iq}=0.1$, the rate of convergence of \eqref{p1lsq} is compared with that of the non-adaptive version of the distributed primal-dual dynamics employed to solve the problem \eqref{dlsq}. The rate of convergence is significantly improved as shown in the Fig. \ref{f22}. The global solution to \eqref{dlsq} is also compared with the solution of the least square solver $\mathrm{lsqlin}$ in $\mathrm{MATLAB}$. The global optimizer $x^*_1=x^*_2=x^*_3=x^*_4$ obtained using the proposed algorithm coincides with the optimal solution $x^*$ obtained using $\mathrm{lsqlin}$ as shown in the Fig. \ref{f19}.
\begin{figure}[!t]
	\centering
	\includegraphics[width=6in]{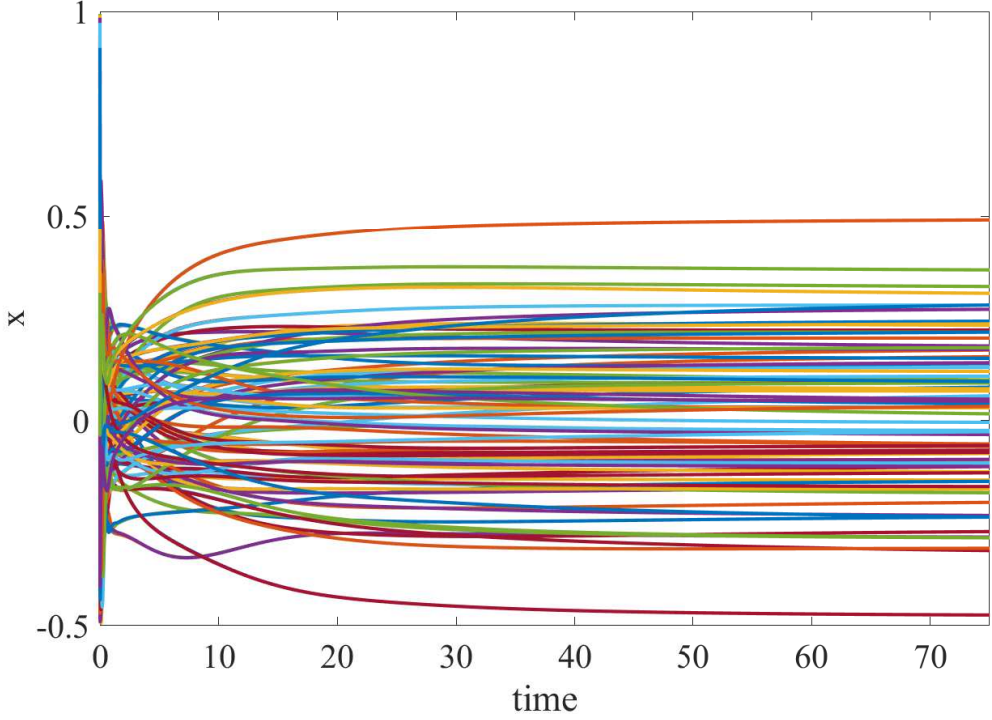}
	\caption{Convergence of \eqref{p1lsq} to the global solution $x^*$.}
	\label{f21}
\end{figure}
\begin{figure}[!t]
	\centering
	\includegraphics[width=6in]{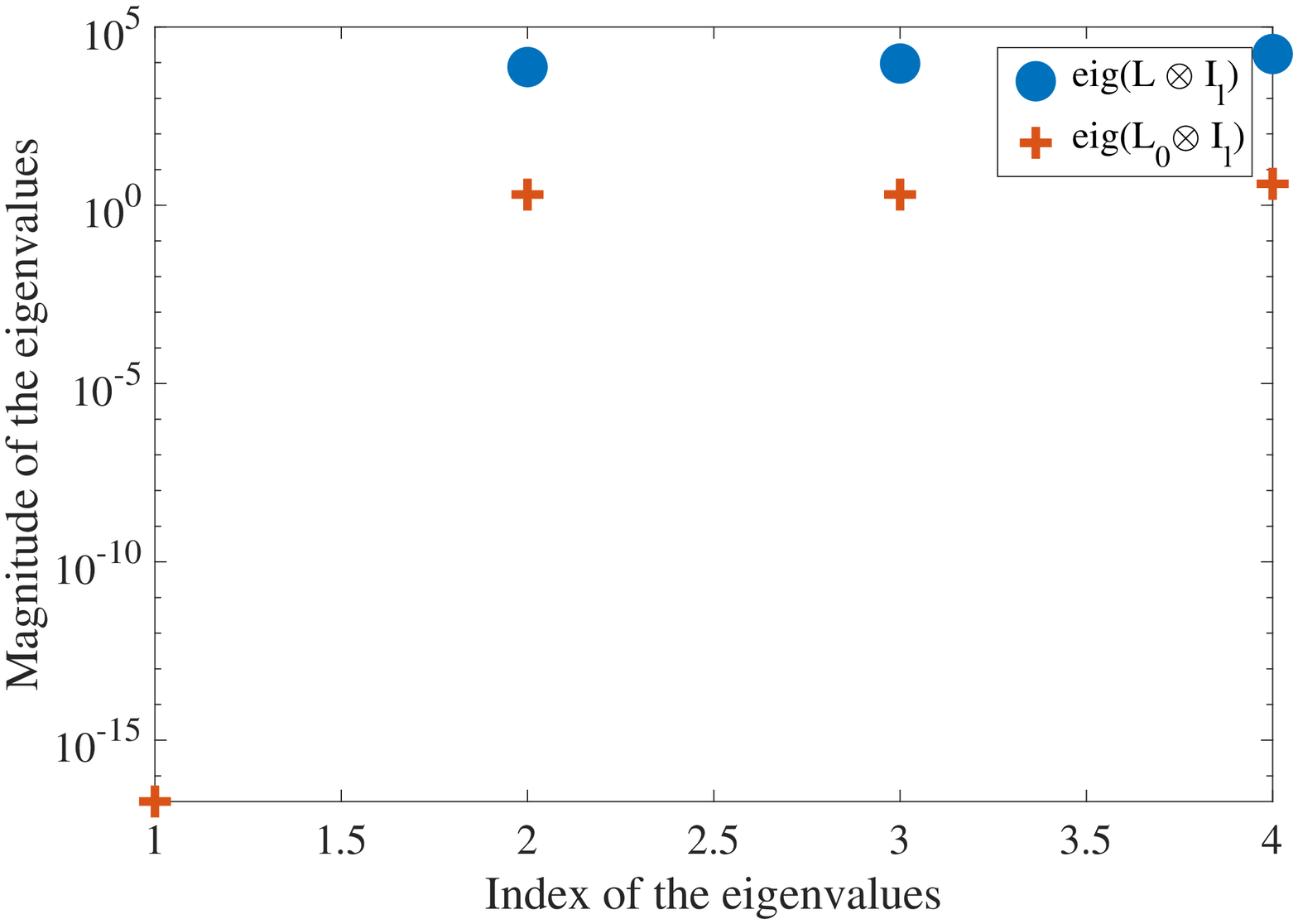}
	\caption{Eigenvalue comparison of $L\otimes I_l$ and $L_0\otimes I_l$ for the problem \eqref{p1lsq}.}
	\label{f22}
\end{figure}
\begin{figure}[!t]
	\centering
	\includegraphics[width=6in]{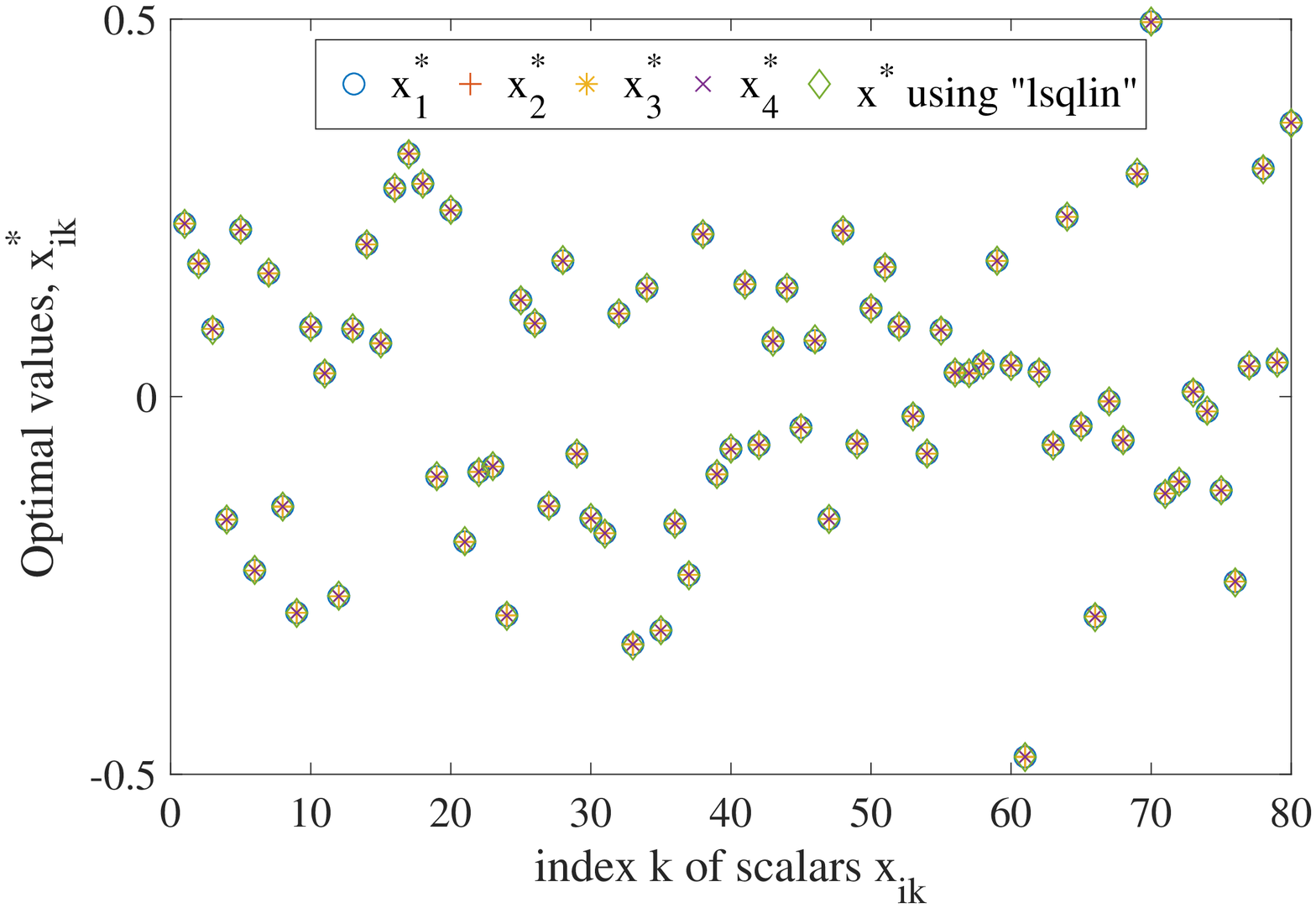}
	\caption{Comparison of the global optimizers of \eqref{dlsq} with the optimal solution of the $\mathrm{lsqlin}$ solver.}
	\label{f19}
\end{figure}
\subsection{Quadratic-inequality Constrained Distributed Least Squares}
A box-constrained linear least squares problem is the one in which the upper and lower bounds on the estimated values are incorporated to handle limitations of the physical system. These methods are studied with applications to GPS positioning \cite{zhu2005bayesian}, geodesic applications \cite{peng2006aggregate,fang2015weighted,wang2016robust} etc. The box-constrained least squares problem is generally defined as follows:
\begin{align}
\min_x \norm{Ax-b}^2, \mathrm{subject~to}~x_l\leq x\leq x_u,\label{blsq}
\end{align} where $x_l$ and $x_u$ are the upper and lower bounds of the variable $x$. It is known that a quadratic constraint formulation of the box constrained least square problem is an efficient approach to obtain the optimal solution of \eqref{blsq} \cite{mead2010least}. The quadratic-constrained equivalent formulation of the box-constrained least square problem \eqref{qlsq} is given as:
\begin{align}
\min_x\norm{Ax-b}^2, \mathrm{subject~to}~(x_i-\bar{x}_i)^2\leq \rho^2_i,\forall^n_{i=1}. \label{qlsq}
\end{align}
where $\bar{x}_i$ is the midpoint of the interval $[x_l,x_u]$. It is computed as $\bar{x}_i=(x_l+x_u)/2$ with $\rho_i = (x_u-x_l)/2$. 

A distributed framework for the quadratic-constrained least squares problem \eqref{qlsq} can be obtained as:
\begin{align}
&\min_x\sum_{i=1}^{n}\norm{A_ix-b_i}^2,\nonumber \\
&\mathrm{subject~to}~(x_{ik}-\bar{x}_{ik})^2\leq \rho^2_{ik},\forall^l_{k=1},\forall^n_{i=1} \nonumber\\
&~~~~~~~~~~~~~~~ x_i = x_j,\forall j\in \mathcal{N}_i.\label{dqlsq}
\end{align}
The ADPDD formulation of the problem \eqref{dqlsq} is similar to that of the proposed dynamics \eqref{h1d}-\eqref{h3}. Hence, it is omitted to avoid repetition of the equations. 
\subsubsection{Simulations}
For the sake of simplicity and readability of the simulation results, a small problem of the form \eqref{dqlsq} is taken as a proof of concept with the parameters $A\in \mathbb{R}^{20 \times 4}$ and $b\in \mathbb{R}^{20 \times 1}$. A network with a cyclic graph topology containing $\mathrm{4}$ agents is considered wherein each agent holds on to $A_i\in \mathbb{R}^{5 \times 4}$ component of the matrix $A$. All agents iteratively reach the global consensus of the optimizer value $x^*$ with $d_{iq}=2$, as shown in the Fig. \ref{f23}. It can be observed that the trajectories $x_1,x_2,x_3$, and $x_4$ synchronize to respective common trajectories at around $t \approxeq 0.03~\mathrm{seconds}$. The result is also compared with the solution of $\mathrm{lsqlin}$ and it can be seen from the Fig. \ref{f25} that the global optimizer of \eqref{dqlsq} coincides with the solution obtained using $\mathrm{lsqlin}$. The accelerated convergence of the proposed algorithm employed to solve \eqref{dqlsq} is evident from the Fig. \ref{f24}.

\begin{remark}
	A strong synchronization between the trajectories of the agents imply guaranteed convergence to the global optimizer under sparse communication events. It is also indicative of the fact that the communication between the agents need not be periodic. The proposed algorithm can be augmented with the event-triggered control framework. 
\end{remark}

\begin{figure}[!t]
	\centering
	\includegraphics[width=6in]{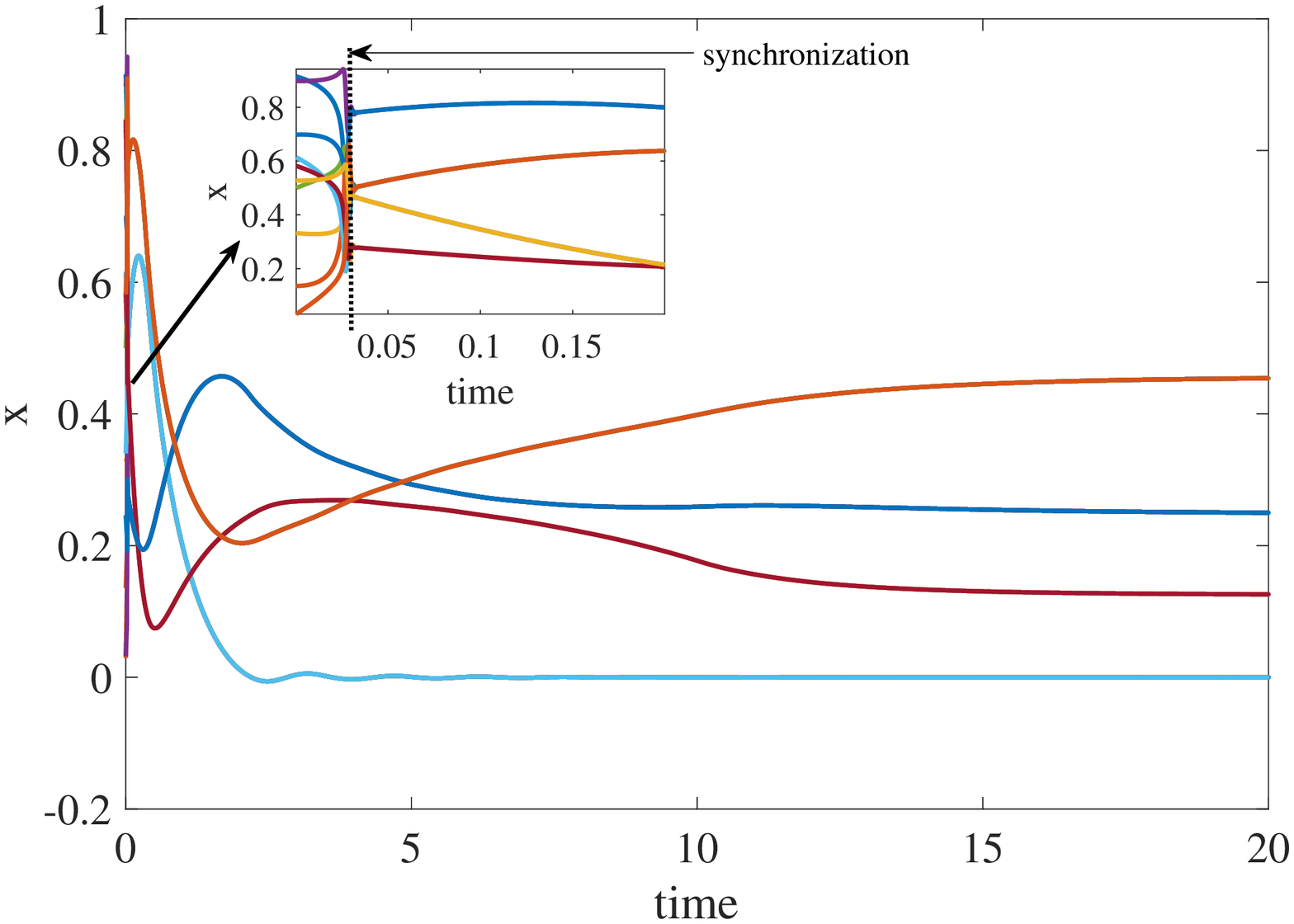}
	\caption{Convergence of the solutions of \eqref{dqlsq} to the global optimizer $x^*$.}
	\label{f23}
\end{figure}
\begin{figure}[!t]
	\centering
	\includegraphics[width=6in]{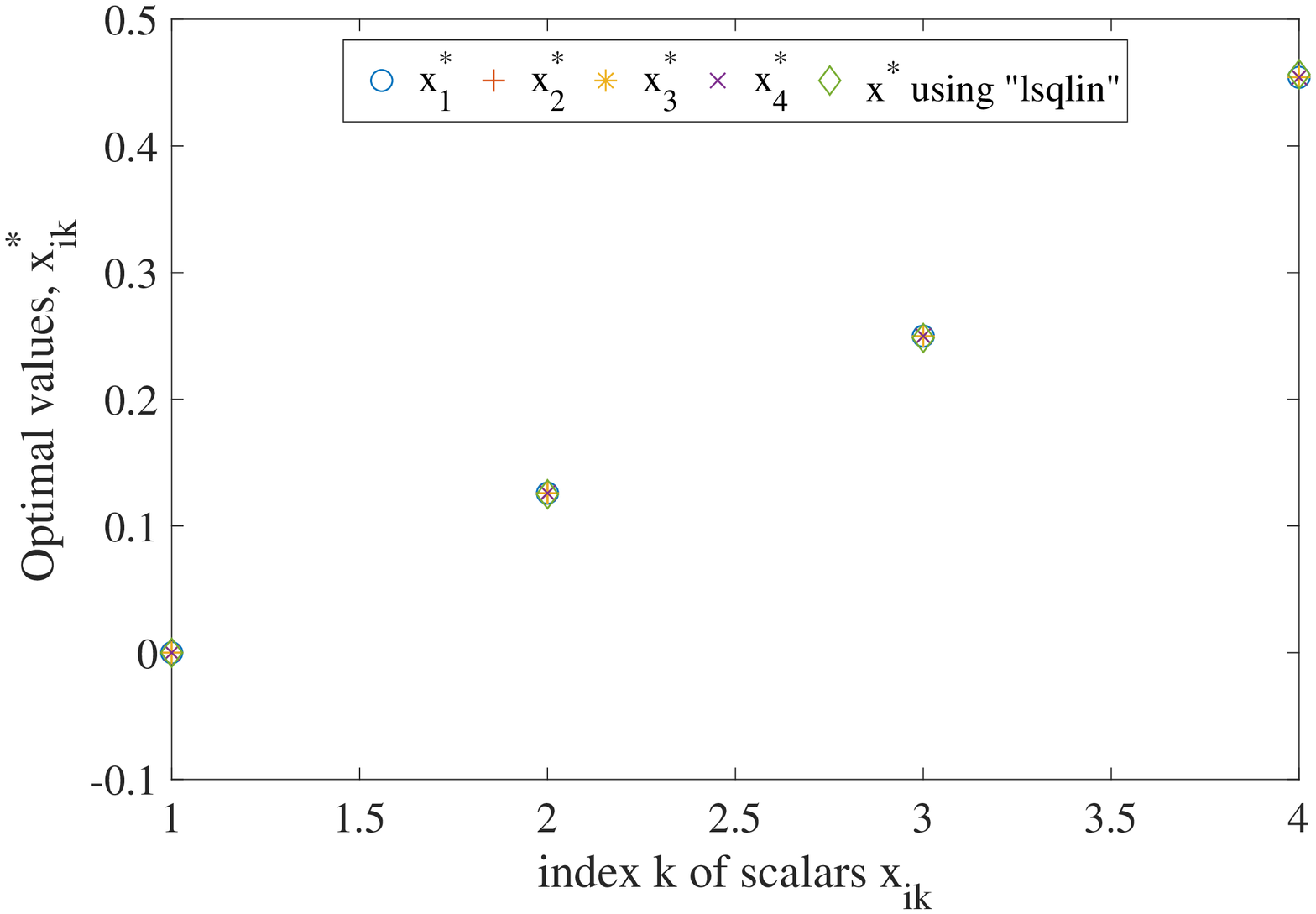}
	\caption{Comparison of the global optimizers of \eqref{dqlsq} with the optimal solution of the $\mathrm{lsqlin}$ solver.}
	\label{f25}
\end{figure}
\begin{figure}[!t]
	\centering
	\includegraphics[width=6in]{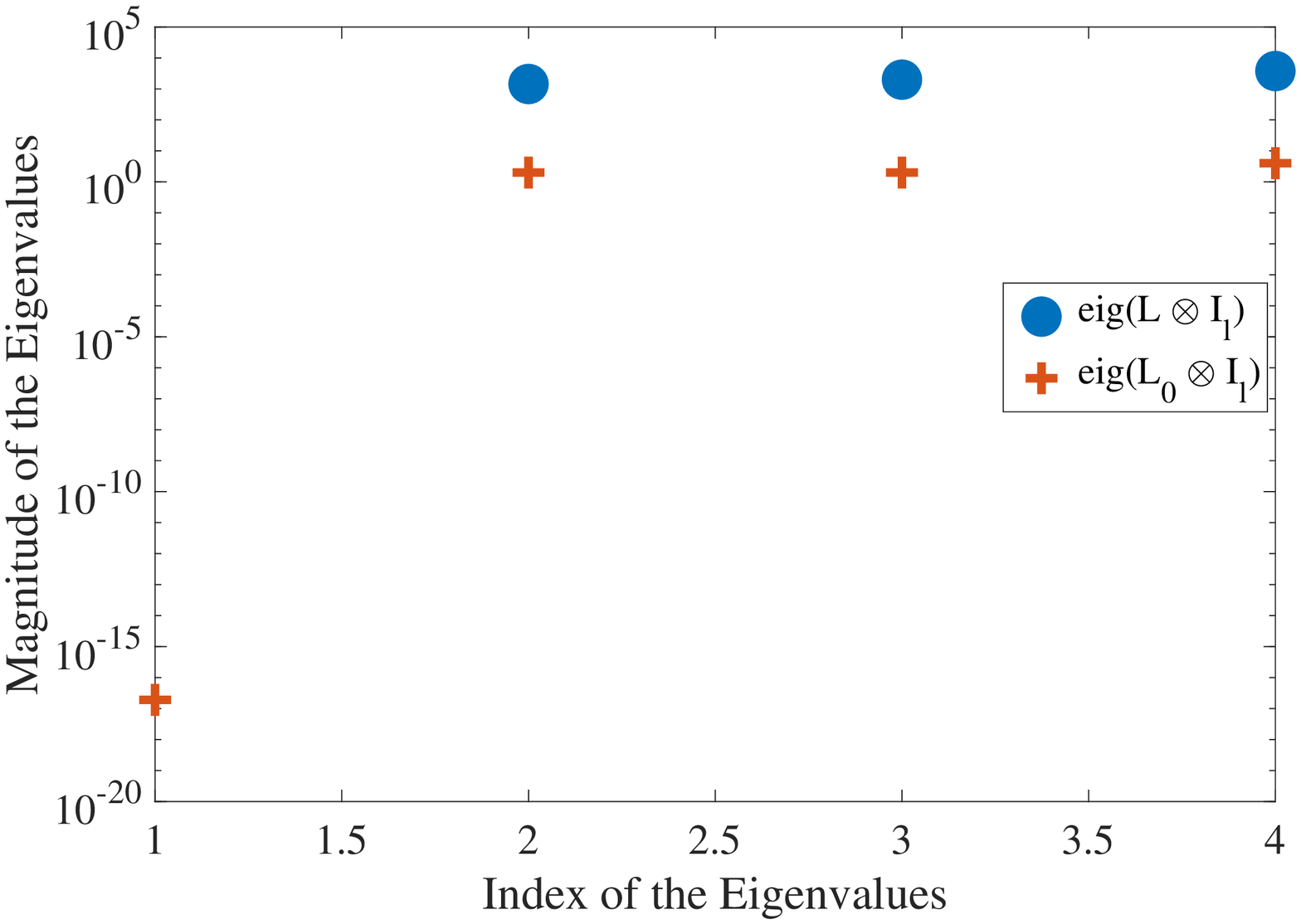}
	\caption{Eigenvalue comparison of $L\otimes I_l$ and $L_0\otimes I_l$ for the problem \eqref{p1lsq}.}
	\label{f24}
\end{figure}
\subsection{Distributed Support Vector Machines}
Support vector machines (SVMs) are supervised learning based paradigms in the machine learning domain, used for classification and regression analysis on raw data, (see \cite{cortes1995support}). For applications with a huge amount of data, there are often limitations with respect to bandwidth requirement, data storage and processing capability of the computing machine, response time, etc. As it turns out, a single computing machine is inefficient in dealing with the SVM algorithm with large datasets. Distributed versions of support vector machines have been proposed as an alternative method to overcome these limitations, as discussed in \cite{forero2010consensus,stolpe2016distributed}. With the aim of enabling accelerated convergence to the optimal solution, the distributed SVM problem is formulated in terms of the adaptive primal-dual dynamics. However, due to the complexity involved with simulations of large-scale SVM problems, the present work only considers the mathematical formulation and does not provide the simulation results for the same.

A problem formulation of the support vector machines for the case of non-separable data is given below:
\begin{align}
\begin{aligned}
&\min~\frac{1}{2}\norm{w}^2+pC\sum_{j=1}^{m}\xi_{j}\\
&\mathrm{s.t.}~ y_j(w^Tx_j+b)\geq 1-\xi_{j},\forall^m_{j=1},    
\end{aligned}\label{cent}
\end{align}
where $\frac{1}{\norm{w}}$ is the margin that separates positive and negative observations, $(x_j,y_j)\in S$ is a paired observation sample, and $w,b$ are weight and bias variables, respectively. $1-\xi_{j}-y_{j}(w^Tx_{j}+b)$ is called as a hinge loss function. C is used to trade off the sum over all slack variables $\xi$ against the size of the margin. $p>0$ is the scaling factor.
\subsubsection{Data Partitioning}
It is assumed that the set of observations $S$ is horizontally partitioned and distributed among computing nodes in $\mathcal{G}(\mathcal{N},\mathcal{E})$ \cite{stolpe2016distributed}, where now $\mathcal{N}= \{1,\ldots,n\}$ represents the computing nodes and the set of edges $\mathcal{E}$ describes communication links between them. Assuming that the graph is connected and enabling only one-hop neighborhood communication, each node $i$ communicates with its neighbors belonging to $\mathcal{N}_i$. Each node $i\in \mathcal{N}$ stores a sample set of labeled observations, denoted by $S_i=\{(x_{i1},y_{i1}),\ldots,(x_{im_i},y_{im_i})\}$.
Note that:
\begin{enumerate}
	\item $S_i$ is a set of labeled observations allocated to $i^{th}$ computing node, $S_i \in S$, where $S$ is a superset of the labeled observations.
	\item $x_{i} \in \mathbb{R}^{m_i\times 1}$.
	\item $y_{ij} \in \{-1,+1\}$ is a class label.
\end{enumerate}
In what follows, an adaptive primal-dual dynamics based formulation of distributed support vector machines is provided. 
\subsubsection{ADPDD formulation of Distributed Support Vector Machines}
A distributed version of the support vector machines problem \eqref{cent} is formulated as given below (see, \cite{forero2010consensus}):
\begin{align}
\begin{aligned}
\min~&\frac{1}{2}\sum_{i=1}^{n}\norm{w_i}^2+pC\sum_{i=1}^{n}\sum_{j=1}^{m_i}\xi_{ij}\\
\mathrm{s.t.}&~y_{ij}(w_ix_{ij}+b_i)\geq 1-\xi_{ij},\xi_{ij}\geq 0,\forall i\in \mathcal{N},\forall^{m_i}_{j=1},\\
&~w_i=w_q,b_i=b_q,~\forall i\in \mathcal{N},q\in\mathcal{N}_i.
\end{aligned}\label{dist}
\end{align}
The objective function in \eqref{dist} is a differentiable $(C^2)$ and strongly convex in $w$. The decision (primal) variables are $w,b \in \mathbb{R}^m$, where $w_i=w_q,b_i=b_q$ are the consensus constraints with $q$ as a neighbor of $i$ if and only if $q \in \mathcal{N}_i$. 
Let $h_{ij}(\xi_{ij},w_i,b_i)=1-\xi_{ij}-y_{ij}(w_ix_{ij}+b_i)$.

The Lagrangian formulation of the problem \eqref{dist} is given by
\begin{align}
\begin{aligned}
\mathcal{L}(w,b,\xi,\theta,\mu,\alpha,\beta)=&\frac{1}{2}\norm{w}^2+pC\sum_{i=1}^{n}\sum_{j=1}^{m_i}\xi_{ij}\\&+\alpha^TLw+\beta^TLb\\&+\sum_{i=1}^{n}\sum_{j=1}^{m_i}\theta_{ij}h_{ij}(\xi_{ij},w_i,b_i)\\&+\sum_{i=1}^{n}\sum_{j=1}^{m_i}\mu_{ij}\xi_{ij}+\frac{1}{2}w^TLw+\frac{1}{2}b^TLb,
\end{aligned}\label{lagr}
\end{align}
where $\theta_{ij},\mu_{ij}$ are the Lagrange multipliers associated with inequality constraints $h_{ij}(\xi_{ij},w_i,b_i)$ and $\xi_{ij}\geq 0$, of $i^{th}$ computing node, and $\alpha_i,\beta_i$ are the Lagrange multipliers associated with coupling constraints of $i^{th}$ and $q^{th}, \forall q \in \mathcal{N}_i$ nodes. $L$ is the Laplacian matrix of the undirected graph $G$. 

Let $z = [w^T, b^T]^T$ (with $z_i=[w_i,b_i]$, $l=2$) then, $e_{iq} = z_i-z_q$. The interconnected network dynamics for the distributed support vector machines problem \eqref{dist} is represented as follows:
\begin{align}
H_1&:\begin{cases}
\dot{w}&=-w- Lw-L\alpha-\zeta,\\
\dot{b}&=- Lb-L\beta-\eta,\\
\dot{a}_{iq}&=d_{iq}(e^T_{iq}e_{iq}+\dot{e}^T_{iq}\dot{e}_{iq}),\forall i \in \mathcal{N},\forall q \in \mathcal{N}_i,\\
u_{H_1}&=-(L \otimes I_l)y_{H_2}-y_{H_3},\\
y_{H_1}&=z.
\end{cases}\label{h1e}
\end{align}
The subsystem $H_2$ contains only consensus-dual variables, with $u_{H_2}$ and $y_{H_2}$ as its input and output respectively, as given below:
\begin{align}
H_2&:\begin{cases}
\dot{\alpha}&=Lw,\\
\dot{\beta}&=Lb,\\
u_{H_2}&=(L \otimes I_l)y_{H_1},\\
y_{H_2}&=[\alpha^T,\beta^T]^T.
\end{cases}\label{h2e}
\end{align}
The subsystem $H_3$ contains the slack variable, and the dual variables corresponding to the inequality constraints, with $u_{H_3}$ and $y_{H_3}$ as its input and output respectively, as given below:
\begin{align}
H_3&:\begin{cases}
\dot{\theta}_{ij}&=[h_{ij}(\xi_{ij},w_i,b_i)]^+_{\theta_{ij}}\forall^{m_i}_{j=1},\forall^n_{i=1},\\
\dot{\mu}_{ij}&=[\xi_{ij}]^+_{\mu_{ij}}\forall^{m_i}_{j=1},\forall^n_{i=1},\\
\dot{\xi}_{ij}&=[-pC-\mu_{ij}+\theta_{ij}]^+_{\xi_{ij}}\forall^{m_i}_{j=1},\forall^n_{i=1},\\
u_{H_3} &= y_{H_1},\\
y_{H_3} &= [\zeta^T,\eta^T]^T,
\end{cases}\label{h13}
\end{align}
where $\zeta, \eta, \mu\in \mathbb{R}^n$, and $\zeta_i = \sum_{j=1}^{m_i}\theta_{ij}(-y_{ij}x_{ij})$ with $\eta_i = \sum_{j=1}^{m_i}\theta_{ij}(-y_{ij})$. 

Thus, the proposed dynamics can be implemented for solving the distributed support vector machines problem \eqref{dist} as shown in \eqref{h1e}-\eqref{h13}. The solution of the underlying dynamics will correspond to the saddle-point solution of \eqref{lagr}, wherein the primal solution is the optimal solution of \eqref{dist}.

In the following, two different formulations of \eqref{diopt} are considered and the results of the proposed dynamics are compared with that of the non-adaptive version of the distributed primal-dual dynamics. 

\subsection{Numerical Example 1} Consider the following distributed optimization problem consisting $\mathrm{3}$ agents having more than one variable and convex inequality constraints.
\begin{align}
\begin{aligned}  
&\min_{x \in \mathbb{R}^6}~\sum_{i=1}^{3}f_i(x_i),\\
&\mathrm{subject~to}~x_i=x_q,~g_i(x_i)\leq 0,~\forall i,q \in \mathcal{N}.
\end{aligned}\label{dex1}
\end{align}
where the objective function associated with each agent is given below
\begin{align}
f_1(x_1)=(x_{11}-x_{12})^2+(x_{11}-1)^2,\\
f_2(x_2)=\frac{1}{3}(x_{21}-x_{22})^2+(x_{21}-3)^2,\\
f_3(x_3)=\frac{1}{3}(x_{31}-x_{32})^2+(x_{31}-6)^2,
\end{align}
with the following local inequality constraints
\begin{align}
g_1(x_1)=6x_{11}^2+3x_{12}^2-11,\\
g_2(x_2)=7x_{11}^2+11x_{12}^2-7,\\
g_3(x_3)=2x_{11}^2+9x_{12}^2-20.
\end{align}
The graph connectivity is assumed to be as follows: $\mathcal{N}_1=1$, $\mathcal{N}_2=2$, and $\mathcal{N}_3=1$. The ADPDD algorithm is employed to solve the problem \eqref{dex1}, and the corresponding trajectories are shown in Fig. \ref{f11}. The primal optimizers are $\mathrm{(1.4099,0.8966)}$. Besides that, in Fig. \ref{f12} the steady state eigenvalues of $L \otimes I_l$ are plotted along with the eigenvalues of $L_0 \otimes I_l$. The eigenvalue $\lambda_2(L \otimes I_l)$ at steady state is equal to $\mathrm{192.8079}$ as compared to the eigenvalue corresponding to a non-adaptive DPDD, $\lambda_2(L_0 \otimes I_l)=1$. From Proposition \ref{lemma12} and Proposition \ref{prope}, it can be seen that the adaptive synchronization has sought to increase the rate of convergence of the ADPDD.
\begin{figure}[!t]
	\centering
	\includegraphics[width=6in]{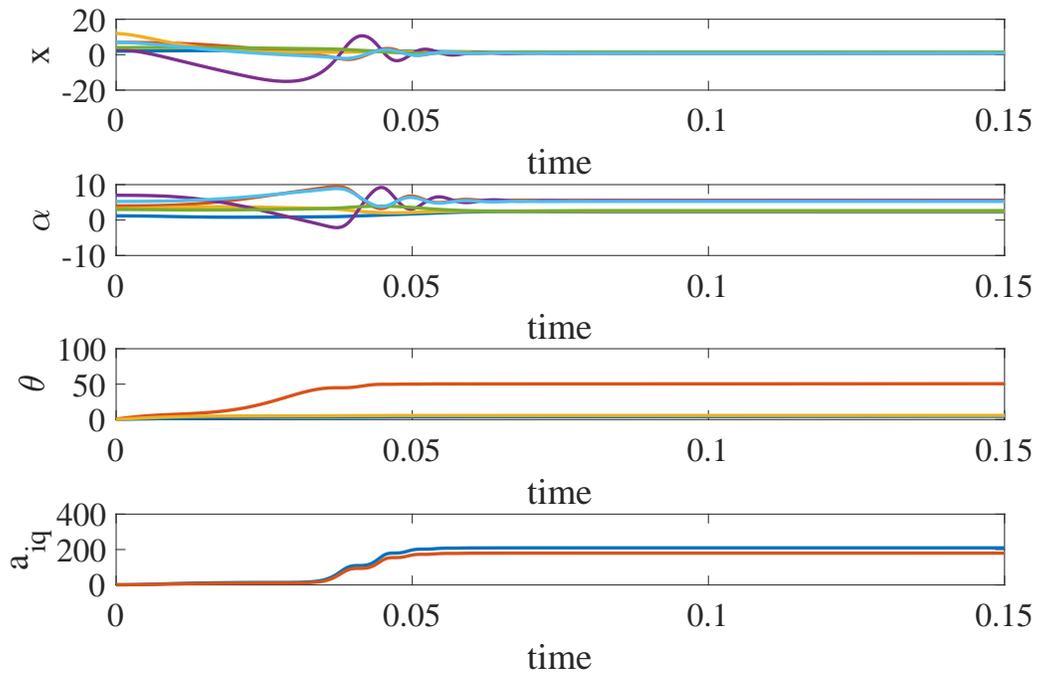}
	\caption{Convergence of the ADPDD (Example 1).}
	\label{f11}
\end{figure}
\begin{figure}[!t]
	\centering
	\includegraphics[width=6in]{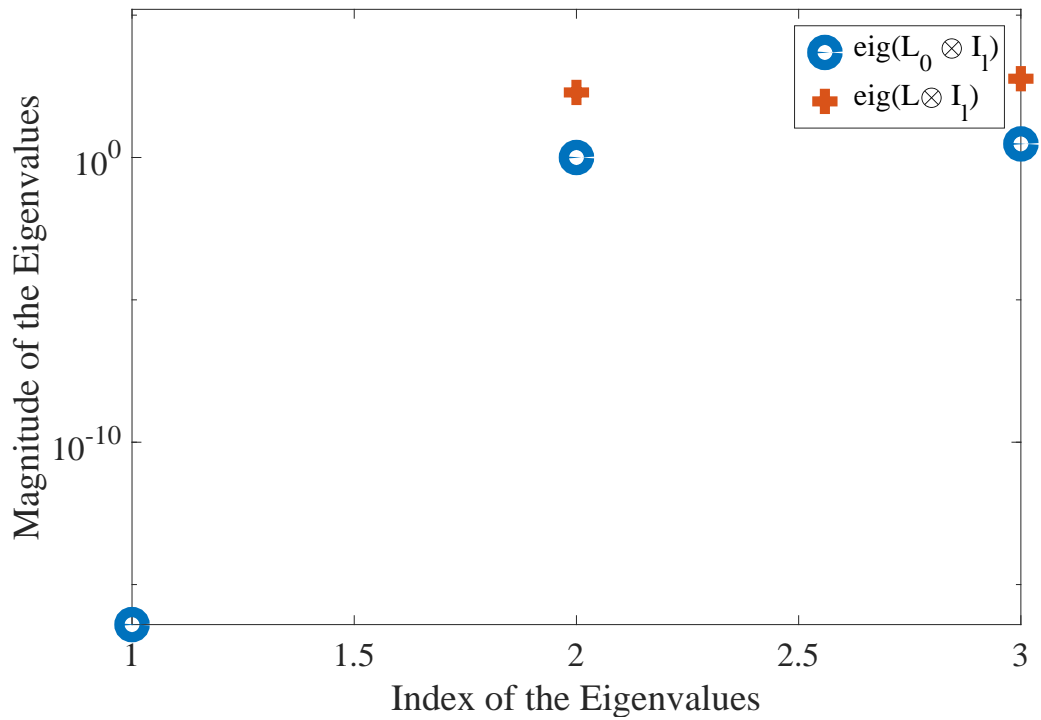}
	\caption{Eigenvalues of $L_0 \otimes I_l$ and $L \otimes I_l$ (Example 1).}
	\label{f12}
\end{figure}
\subsection{Numerical Example 2}\label{ex3} In this subsection, the local inequality constraints associated with each agent are relaxed and the following optimization problem is considered on a random graph with $\mathrm{10}$ agents as shown in Fig. \ref{f13}. Note that the degree of each agent is selected randomly.
\begin{align}
\begin{aligned}  
&\min_{x \in \mathbb{R}^{10}}~\sum_{i=1}^{10}f_i(x_i),\\
&\mathrm{subject~to}~x_i=x_q,~\forall i,q \in \mathcal{N}.
\end{aligned}\label{dex2}
\end{align} with a randomly generated Hessian $\mathbb{H}=\mathrm{diag}([136,439,355,298,302,350,327,398,353,294])$.
The proposed dynamics is employed to solve \eqref{dex2}, first considering $d_{iq}=0.001$ and then $d_{iq}=0.01$. Fig. \ref{f14} and Fig. \ref{f15} correspond to the case of $d_{iq}=0.001$ while Fig. \ref{f114} and Fig. \ref{f115} correspond to the case of $d_{iq}=0.01$. It can be seen that for the latter case the convergence is much faster. This owes to the difference between the resulting eigenvalues, i.e., for the case of $d_{iq}=0.001$, the second smallest eigenvalue $\lambda_2(L \otimes I_l)$ yields to be $\mathrm{10.72}$ whereas the same for the case of $d_{iq}=0.01$ increases to $\mathrm{33310}$. The eigenvalue results for both values of $d_{iq}$ are shown in the Fig. \ref{f15} and the Fig. \ref{f115}.
\begin{figure}[!t]
	\centering
	\includegraphics[width=6in]{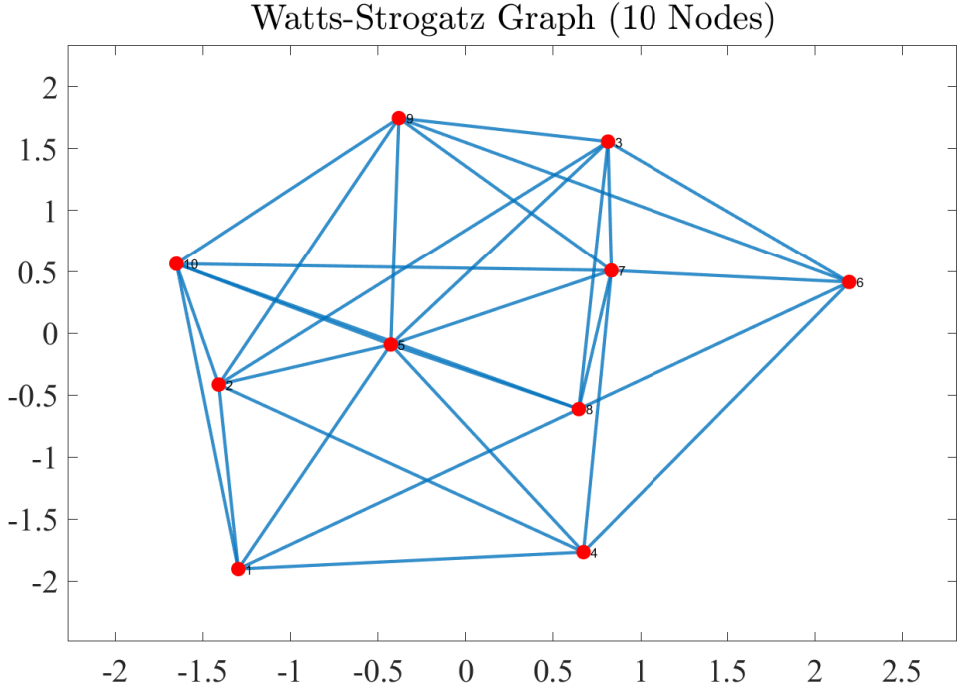}
	\caption{A random graph $\mathcal{G}$ containing $\mathrm{10}$ nodes (Example 2).}
	\label{f13}
\end{figure}
\begin{figure}[!t]
	\centering
	\includegraphics[width=6in]{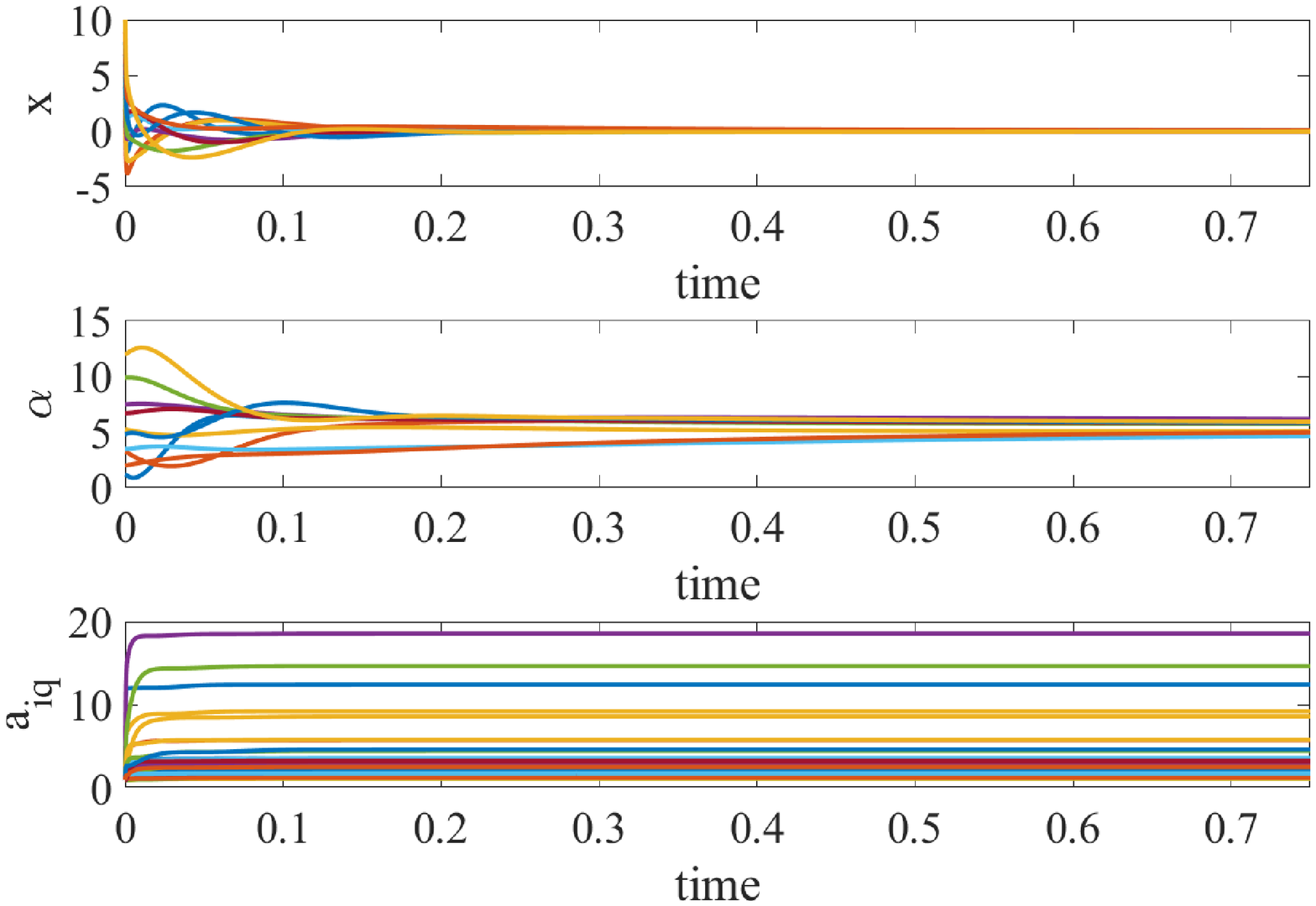}
	\caption{Convergence of the ADPDD ($d_{iq}=0.001$) (Example 2).}
	\label{f14}
\end{figure}
\begin{figure}[!t]
	\centering
	\includegraphics[width=6in]{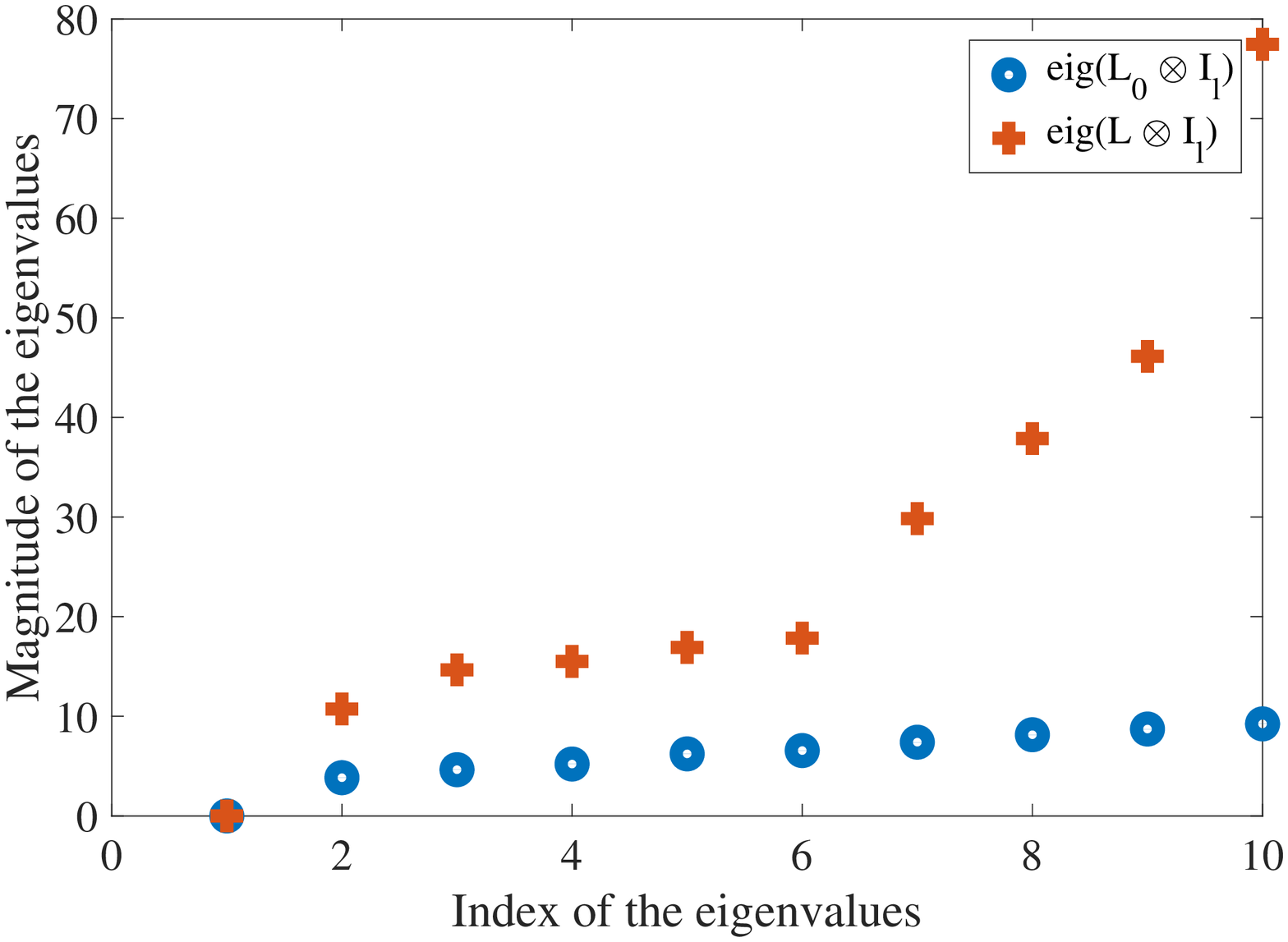}
	\caption{Eigenvalues of $L_0\otimes I_l$ and $L \otimes I_l$ ($d_{iq}=0.001$) (Example 2).}
	\label{f15}
\end{figure}

\begin{figure}[!t]
	\centering
	\includegraphics[width=6in]{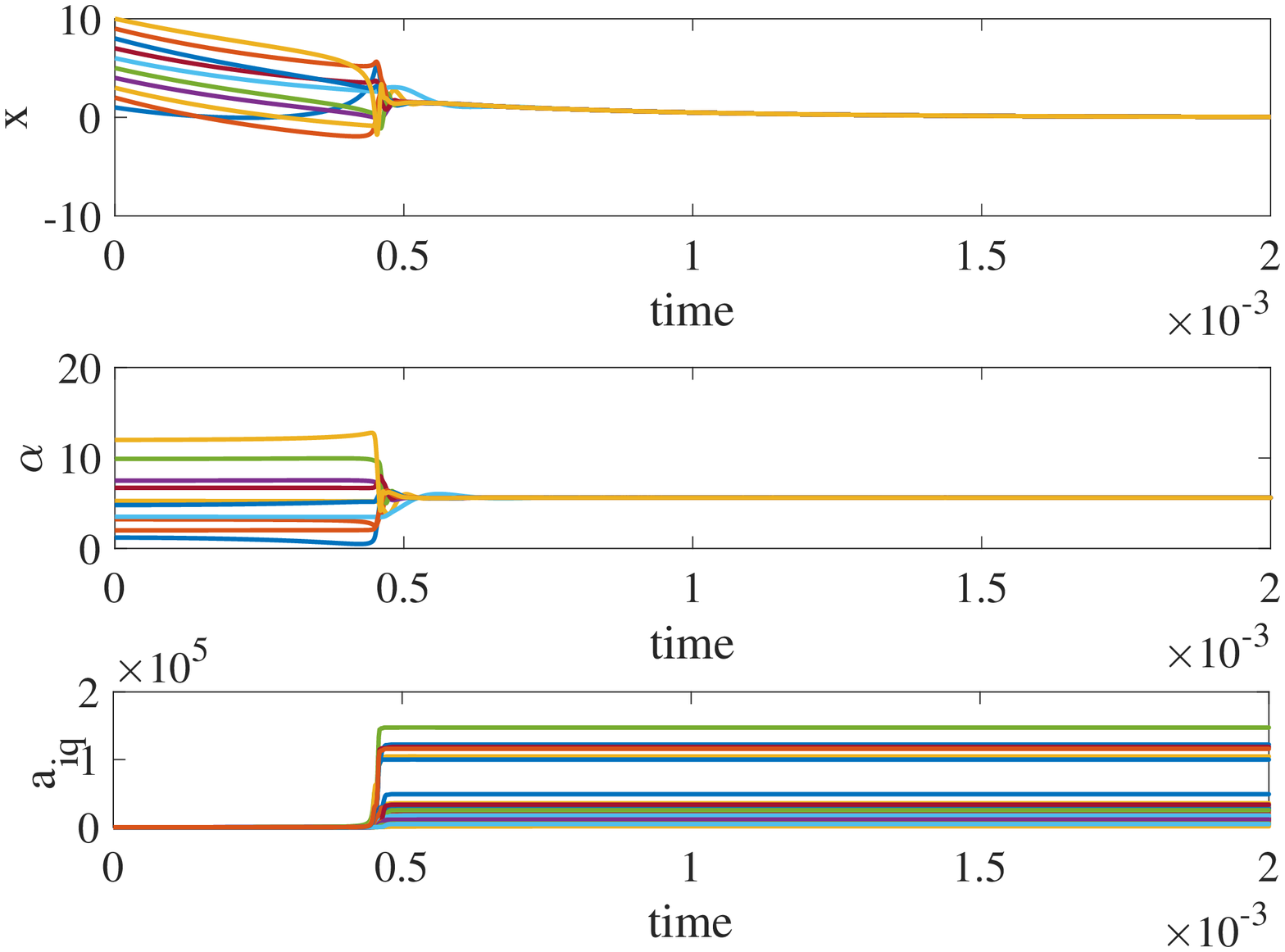}
	\caption{Convergence of the ADPDD ($d_{iq}=0.01$) (Example 2).}
	\label{f114}
\end{figure}
\begin{figure}[!t]
	\centering
	\includegraphics[width=6in]{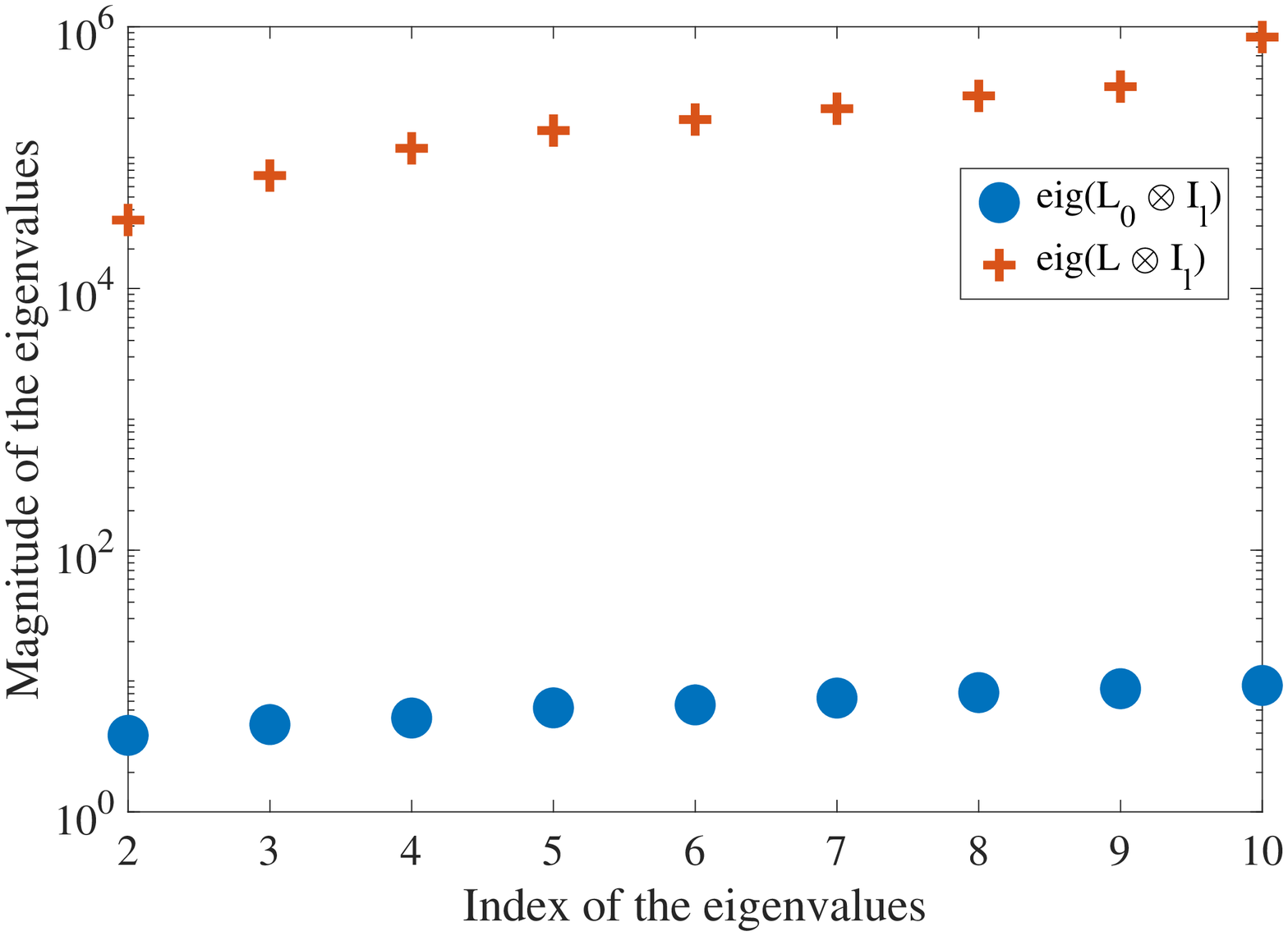}
	\caption{Eigenvalues of $L_0 \otimes I_l$ and $L \otimes I_l$ ($d_{iq}=0.01$) (Example 2).}
	\label{f115}
\end{figure}


\section{Conclusions}\label{concl}
In this paper, an adaptive distributed primal-dual dynamics is proposed to solve inequality and consensus constrained distributed optimization problems. The adaptive synchronization of the primal variables is brought into play by allowing the coupling weights to update according to the difference between the local trajectories (trajectories belonging to the neighboring nodes or agents)  as well as the difference between the rate of change of the local trajectories respectively. It is proved that the proposed dynamics represents a network of feedback-interconnected passive dynamical systems which are asymptotically stable. Further, by allowing a time-scale separation between the adaptive coupling law and primal dynamics, stronger convergence bounds for the primal dynamic are derived, and it is proved that the adaptively coupled primal dynamics converges to the unique primal optimizer.

The performance of the proposed dynamics is quantified in terms of the induced $L_2$-gain from the disturbance input to the output. The effect of adaptive synchronization on the $L_2$-gain is discussed and it is established that the adaptive distributed primal-dual dynamics are comparatively less robust to the exogenous input disturbances than the distributed primal-dual dynamics. On the other hand, the analysis also revealed that in order to achieve accelerated convergence to the saddle-point solution, the proposed algorithm must call for a trade-off between the convergence and the robustness parameters. 

The future scope of the work will be directed towards improving the rate of convergence of the proposed dynamics without compromising its robustness properties. Its applications to large-scale distributed optimization problems such as distributed support vector machines \cite{stolpe2016distributed}, distributed least squares \cite{wang2019distributed} etc will be considered.

\bibliographystyle{unsrt}  
\bibliography{references}  


\end{document}